\documentclass[12pt,twoside,a4paper]{article}

\usepackage{amsmath,amsthm,amssymb,tikz-cd,hyperref}
\allowdisplaybreaks

\newtheorem{theorem}{Theorem}[section]					
\newtheorem{lemma}[theorem]{Lemma}
\newtheorem{corollary}[theorem]{Corollary}
\newtheorem{proposition}[theorem]{Proposition}
\theoremstyle{definition}
\newtheorem{definition}[theorem]{Definition}
\theoremstyle{definition}
\newtheorem{example}[theorem]{Example}

\newcommand{\N}{\mathbb{N}}								
\newcommand{\Z}{\mathbb{Z}}
\newcommand{\Q}{\mathbb{Q}}

\newcommand{\C}{\mathbb{C}}
\newcommand{\F}{\mathbb{F}}
\newcommand{\K}{\mathbb{K}}
	
\newcommand{\CC}{\mathcal{C}}							
\newcommand{\DD}{\mathcal{D}}
\newcommand{\FF}{\mathcal{F}}						
\newcommand{\OO}{\mathcal{O}}
\newcommand{\RR}{\mathcal{R}}
\renewcommand{\SS}{\mathcal{S}}

\newcommand{\catname}[1]{{\normalfont\textbf{#1}}}		

\newcommand{\biset}[2]{{}_{#1}\catname{set}_{#2}}

\newcommand{\lmod}[1]{{}_{#1}\catname{mod}}
\newcommand{\lMod}[1]{{}_{#1}\catname{Mod}}
\newcommand{\lset}[1]{{}_{#1}\catname{set}}

\newcommand{\rset}[1]{\catname{set}_{#1}}

\newcommand{\Aut}{\operatorname{Aut}}

\newcommand{\Def}{\operatorname{Def}}
\newcommand{\Defres}{\operatorname{Defres}}

\newcommand{\dsum}{\oplus}
\newcommand{\Dsum}{\bigoplus}
\newcommand{\End}{\operatorname{End}}

\newcommand{\Gal}{\operatorname{Gal}}

\newcommand{\gp}[1]{\langle#1\rangle}
\newcommand{\Hom}{\operatorname{Hom}}

\newcommand{\id}{\operatorname{id}}

\newcommand{\im}{\operatorname{im}}
\newcommand{\Ind}{\operatorname{Ind}}
\newcommand{\Indinf}{\operatorname{Indinf}}
\newcommand{\Inf}{\operatorname{Inf}}

\newcommand{\into}{\hookrightarrow}
\newcommand{\Irr}{\operatorname{Irr}}
\newcommand{\iso}{\cong}
\newcommand{\isoto}{\overset{\sim}{\to}}
\newcommand{\Iso}{\operatorname{Iso}}

\newcommand{\lin}{\operatorname{lin}}

\newcommand{\Mult}{\operatorname{Mult}}
\newcommand{\nor}{\trianglelefteq}

\newcommand{\onto}{\twoheadrightarrow}
\newcommand{\op}[1]{#1^{\text{op}}}
\newcommand{\Out}{\operatorname{Out}}
\newcommand{\rank}{\operatorname{rank}}
\newcommand{\Res}{\operatorname{Res}}
\newcommand{\semil}{\rtimes}

\newcommand{\set}[1]{\left\{#1\right\}}

\newcommand{\spn}{\operatorname{span}}

\newcommand{\subgp}{\leq}

\newcommand{\tensor}{\otimes}

\newcommand{\Tor}{\operatorname{Tor}}

	\makeatletter
	\newcommand{\tpitchfork}{%
		\vbox{
			\baselineskip\z@skip
			\lineskip-.52ex
			\lineskiplimit\maxdimen
			\m@th
			\ialign{##\crcr\hidewidth\smash{$-$}\hidewidth\crcr$\pitchfork$\crcr}
		}%
	}
	\makeatother

\newcommand{\RES}{\underline{\Res}}
\newcommand{\IND}{\underline{\Ind}}

\title{On the image of the trivial source ring in the ring of virtual characters of a finite group}
\author{John Revere McHugh}

\begin{document}

\maketitle

\begin{abstract}
	We examine the cokernel of the canonical homomorphism from the trivial source ring of a finite group to the ring of $p$-rational complex characters. We use Boltje and Co\c{s}kun's theory of fibered biset functors to determine the structure of the cokernel. An essential tool in the determination of this structure is Bouc's theory of rational $p$-biset functors.
\end{abstract}

\section{Introduction}\label{sec:intro}

Associated to a finite group $G$ are various representation rings which encode the actions of the group on sets, vector spaces, or other structures. Typical examples are the virtual character ring $R_{\C}(G)$, the Burnside ring $B(G)$, and the trivial source ring $T_{\OO}(G)$ (for $\OO$ a complete discrete valuation ring). The theory of \textit{biset functors}, introduced by Bouc in \cite{Bouc_1996} and \cite{Bouc_2010}, provides a means of unifying the constructions $G\mapsto R_{\C}(G)$, $G\mapsto B(G)$, etc. by realizing each as an additive functor defined on the \textit{biset category} $\CC$. The objects of $\CC$ are finite groups, and $\CC$ is generated as a preadditive category by five types of morphisms known as \textit{induction}, \textit{restriction}, \textit{inflation}, \textit{deflation}, and \textit{isomorphism}. The images of these morphisms under a given representation ring (viewed now as a functor defined on $\CC$) are the usual operations from the representation theory of finite groups. In \cite{Boltje_2018} Boltje and Co\c{s}kun expand this framework by defining, for a fixed abelian group $A$, the $A$\textit{-fibered biset category} $\CC^{A}$ and $A$\textit{-fibered biset functors} defined on it. Endowing a biset functor with the extra structure of fibered biset functor grants an additional type of operation: informally speaking, that of ``multiplication'' by a 1-dimensional $A$-character.

Several fundamental theorems have reinterpretations within this framework. For example, recall that for each finite group $G$ there is a \textit{linearization morphism} $\lin_{G}:B^{\C^{\times}}(G)\to R_{\C}(G)$ from the $\C^{\times}$-monomial Burnside ring to the virtual character ring. It turns out that the maps $\lin_{G}$ are the components of a natural transformation $\lin:B^{\C^{\times}}\to R_{\C}$ between $\C^{\times}$-fibered biset functors, and Brauer's Induction Theorem is equivalent to the statement that this natural transformation is surjective. From this perspective, some questions that arise are: (1) given a pair of fibered biset functors, what are the natural transformations between them?  (2) given a natural transformation of fibered biset functors, what is its image/kernel/cokernel, etc.?

In this note we examine the cokernel of a certain natural transformation $\kappa$ from the trivial source ring to a subring of the virtual character ring. Let us briefly recall the definitions. If $\OO$ is a complete discrete valuation ring with field of fractions $\K$ of characteristic 0 and residue field $k$ of positive characteristic $p$ then the triple $(\K,\OO,k)$ is called a $p$\textit{-modular system}. Assume that $k$ is algebraically closed. If $G$ is a finite group, recall that a finitely generated $\OO G$-module is a \textit{trivial source} (or $p$\textit{-permutation})\textit{ module} if it is isomorphic to a direct summand of a permutation $\OO G$-module. The Grothendieck ring of the category of trivial source modules, with respect to split short exact sequences, is denoted $T_{\OO}(G)$ and is called the \textit{trivial source ring} (see \cite[Section~5.5]{Benson_1991_vol1}). We also have $R_{\K}(G)$, the Grothendieck ring of the category of finitely generated $\K G$-modules. Extension of scalars from $\OO$ to $\K$ induces for each finite group $G$ a homomorphism $\kappa_{G}:T_{\OO}(G)\to R_{\K}(G)$, and in fact the maps $\kappa_{G}$ are the components of a natural transformation $\kappa:T_{\OO}\to R_{\K}$ between fibered biset functors. Let $\mu_{p'}$ denote the group of complex roots of unity whose orders are not divisible by $p$. Then because of the assumption on $k$ we can identify $K=\Q(\mu_{p'})$ as a subfield of $\K$, and we can identify $R_{K}(G)$ as a subring of $R_{\K}(G)$. By a theorem of Dress (specifically, \cite[Theorem~1]{Dress_1975}), for each finite group $G$ the image of $\kappa_{G}$ is contained in $R_{K}(G)$. Thus we may regard $\kappa$ as a natural transformation from $T_{\OO}$ to $R_{K}$. 

Our motivation for studying the cokernel $R_{K}/\im(\kappa)$ comes from recent work of Boltje and Perepelitsky. Let $G$ and $H$ be finite groups, let $A$ be a block algebra of $\OO G$, and let $B$ be a block algebra of $\OO H$. In \cite{Boltje_2020} Boltje and Perepelitsky define a $p$\textit{-permutation equivalence} between $A$ and $B$ to be an element $\gamma\in T^{\Delta}(A,B)$, the subgroup of the Grothendieck group of trivial source $(A,B)$-bimodules spanned by the isomorphism classes of indecomposable trivial source bimodules having twisted diagonal vertices, that satisfies
\begin{equation*}
	\gamma\underset{H}{\cdot}\gamma^{\circ}=[A]\in T^{\Delta}(A,A)\qquad\text{and}\qquad\gamma^{\circ}\underset{G}{\cdot}\gamma=[B]\in T^{\Delta}(B,B).
\end{equation*}
Here $\gamma^{\circ}$ denotes the $\OO$-dual of $\gamma$ and $\underset{H}{\cdot}$ is the map induced by the tensor product over $\OO H$. The authors show (see \cite[Theorem~1.6 and Theorem~15.4]{Boltje_2020}) that a $p$-permutation equivalence $\gamma\in T^{\Delta}(A,B)$ induces an isotypy between $A$ and $B$. The perfect isometries that make up this induced isotypy are elements of $\im(\kappa_{C})$ for various groups $C$. Thus if one wishes to lift a given isotypy to a $p$-permutation equivalence one must first check that the perfect isometries of the isotypy are contained in the image of $\kappa$. In other words, the cokernel $R_{K}/\im(\kappa)$ measures a first obstruction to the problem of lifting an isotypy to a $p$-permutation equivalence.

It turns out that the image of $\kappa$ almost always coincides with $R_{K}$ (see Corollary \ref{cor:quotientisF2space}):
\begin{theorem}\label{thm:thmforintro}
	If $p\neq 2$ then $\im(\kappa)=R_{K}$. If $p=2$ then for any finite group $G$ the quotient $R_{K}(G)/\im(\kappa_{G})$ is a finite-dimensional vector space over $\F_{2}$.
\end{theorem}
\noindent Thus when $p=2$ the cokernel $R_{K}/\im(\kappa)$ is a nonzero fibered biset functor for the fiber group $A=\mu_{2'}$. In Theorem \ref{thm:subfunctorsofRK/imkappa} we determine the complete subfunctor structure of this quotient. As part of this classification it is shown that $R_{K}/\im(\kappa)$ is \textit{generated} as a fibered biset functor by the unique faithful irreducible character of the quaternion group $Q_{8}$. If ``the problem'' with the prime $2$ is that $\im(\kappa)\neq R_{K}$, then in some sense this result identifies $Q_{8}$ as the source of the problem. We also provide, in Corollary \ref{cor:detectionthm}, a necessary and sufficient condition for a character $\chi\in R_{K}(G)$ to be an element of $\im(\kappa_{G})$.

The characterization of the subfunctors of $R_{K}/\im(\kappa)$ makes use of Bouc's theory of \textit{rational} $p$\textit{-biset functors} in an essential way. These are functors that are defined on the full subcategory $\CC_{p}$ of the biset category whose objects are $p$-groups and behave like the functor $R_{\Q}$ of rational representions (we recall the exact definition in Section \ref{sec:rationalpbisetfunctors}). The class of rational $p$-biset functors is closed under taking subfunctors and taking quotients, and the subfunctors of a fixed rational $p$-biset functor have been completely classified by Bouc (see \cite[Section~10.2]{Bouc_2010}). In Section \ref{sec:subfunctorsofRKmodimkappa} we show that the restriction of $R_{K}/\im(\kappa)$ to $\CC_{p}$ is a rational $p$-biset functor; in fact, it is equal to $\overline{R}_{\Q}/R_{\Q}$ (here, if $G$ is a finite group then $\overline{R}_{\Q}(G)$ denotes the ring of $\Q$-valued virtual characters of $G$. In subsection \ref{subsec:examples} we will see that $\overline{R}_{\Q}$ has a biset functor structure, hence restricts to a $p$-biset functor). Using rationality, we determine the subfunctors of $\overline{R}_{\Q}/R_{\Q}$ in Theorem \ref{thm:subfunctorsofRbarQ/RQ}. We then find, through \textit{induction} of fibered biset functors, that the subfunctor structure of $R_{K}/\im(\kappa)$ is identical to that of $\overline{R}_{\Q}/R_{\Q}$. In particular, the nonzero subfunctors (in the case $p=2$) are generated by characters of the generalized quaternion groups and both functors are uniserial. Moreover, Theorem \ref{thm:thmforintro} reflects a theorem of Roquette, which is given in \cite[Corollary~10.14]{Isaacs_1976}: if $p\neq 2$ then $\overline{R}_{\Q}(P)=R_{\Q}(P)$ for any finite $p$-group $P$ and if $p=2$ then the quotient $\overline{R}_{\Q}(P)/R_{\Q}(P)$ is a finite-dimensional vector space over $\F_{2}$.

The author would like to thank his advisor, Robert Boltje, for his constant support and for donating generous amounts of his time toward the completion of this project. 

\section{Characters and Schur Indices}

We begin by setting some of the notation that will be used throughout. Let $G$ be a finite group and let $K$ be a field of characteristic 0. The category of finitely generated left $KG$-modules is denoted $\lmod{KG}$. The corresponding Grothendieck ring, with respect to short exact sequences, is denoted $R_{K}(G)$. If $V\in\lmod{KG}$ we write $[V]$ for the isomorphism class of $V$ and, abusively, the image of $V$ in $R_{K}(G)$. The ring $R_{K}(G)$ is identified with the virtual character ring of $KG$, which is the subring of the ring of $K$-valued class functions on $G$ generated by the characters of the irreducible $KG$-modules. The set $\Irr_{K}(G)$ of irreducible characters of $KG$ is a $\Z$-basis of $R_{K}(G)$. If $L$ is an extension of $K$ then extension of scalars from $K$ to $L$ induces an injective ring homomorphism $R_{K}(G)\into R_{L}(G)$. Since this homomorphism is compatible with the identification of $R_{K}(G)$ with the virtual character ring of $KG$ we identify $R_{K}(G)$ as a subring of $R_{L}(G)$.

Let $L\supseteq K$ be an extension of fields. If $\chi\in\Irr_{L}(G)$ then $K(\chi)$ denotes the subfield of $L$ generated by $K$ and the character values $\chi(g)$, $g\in G$. If also $\psi\in\Irr_{L}(G)$ then $\chi$ and $\psi$ are said to be \textit{Galois conjugate over }$K$ if $K(\chi)=K(\psi)$ and if there exists an automorphism $\sigma\in\Gal(K(\chi)/K)$ such that $\sigma(\chi(g))=\psi(g)$ for all $g\in G$. In this case we write ${}^{\sigma}\chi=\psi$. Keeping $K$ and $L$ fixed any $\chi\in\Irr_{L}(G)$ determines a \textit{Galois class sum} $\overline{\chi}$, defined to be the sum of the distinct Galois conjugates of $\chi$ over $K$.

When $L$ is algebraically closed the ring of $K$-valued elements of $R_{L}(G)$ is denoted $\overline{R}_{K}(G)$. If $\chi_{1},\chi_{2},\ldots,\chi_{r}$ are representatives for the equivalence classes of $\Irr_{L}(G)$ with respect to Galois conjugacy over $K$ then the Galois class sums $\overline{\chi_{i}}$, $1\leq i\leq r$, form a $\Z$-basis of $\overline{R}_{K}(G)$. Now if $\chi\in\Irr_{L}(G)$ then some positive multiple $m\chi$ of $\chi$ is the character of a $K(\chi)G$-module. The least $m$ for which this holds is the \textit{Schur index of }$\chi$\textit{ over }$K$ and is denoted $m_{K}(\chi)$.\footnote{A thorough exposition of the theory of Schur indices may be found in \cite[Chapters 9 and 10]{Isaacs_1976}.} If $\theta\in\Irr_{K}(G)$ then $\theta=m_{K}(\chi_{i})\overline{\chi_{i}}$ for some unique $\chi_{i}$ among the representatives above. Furthermore, any character of the form $m_{K}(\chi_{i})\overline{\chi_{i}}$ is afforded by a (unique, up to isomorphism) irreducible $KG$-module. Therefore the characters $m_{K}(\chi_{i})\overline{\chi_{i}}$, $1\leq i\leq r$, form a $\Z$-basis of $R_{K}(G)$. Consequently, as abelian groups we have
\begin{equation}\label{eqn:RbarmodR}
\overline{R}_{K}(G)/R_{K}(G)\iso\prod_{i=1}^{r}\Z/m_{K}(\chi_{i})\Z.
\end{equation}

For the remainder of this section we take $L=\C$. For ease, set $\Irr(G)=\Irr_{\C}(G)$.

\begin{lemma}\label{lem:roquette}
	Let $P$ be a $p$-group, where $p\in\N$ is prime, and let $\chi\in\Irr(P)$. If $m_{\Q}(\chi)\neq 1$ then $p=2$, and if $p=2$ then $m_{\Q}(\chi)=1$ or $2$. In particular, $R_{\Q}(P)\neq\overline{R}_{\Q}(P)$ only when $P$ is a $2$-group, and in this case $\overline{R}_{\Q}(P)/R_{\Q}(P)$ is a (finite-dimensional) vector space over $\F_{2}=\Z/2\Z$.
\end{lemma}

\begin{proof}
	The first part of the lemma is just a particular case of a result of Roquette \cite{Roquette_1958} --- a proof is also given in \cite[Corollary~10.14]{Isaacs_1976}. The second part follows from the first and the isomorphism given in (\ref{eqn:RbarmodR}), above.
\end{proof}

\begin{lemma}\label{lem:schurindicescyclicdihedral}
	Let $P$ be a $2$-group. If $P$ is cyclic, dihedral, or semidihedral then $m_{\Q}(\chi)=1$ for all $\chi\in\Irr(P)$.
\end{lemma}

\begin{proof}
	If $P$ is cyclic this is a consequence of Wedderburn's Theorem. A proof for the case where $P$ is dihedral or semidihedral is given in \cite[(11.7)]{Feit_1967}.
\end{proof}

We now describe a part of the representation theory of the generalized quaternion group $Q_{2^{n}}$. Recall that $Q_{2^{n}}$ is the group of order $2^{n}$, $n\geq 3$, defined by the presentation
\begin{equation*}
Q_{2^{n}}=\gp{a,b|a^{2^{n-1}}=1,b^{2}=a^{2^{n-2}},a^{b}=a^{-1}}.
\end{equation*}
The center $Z(Q_{2^{n}})$ of $Q_{2^{n}}$ is generated by $b^{2}$ and is the unique minimal nontrivial subgroup of $Q_{2^{n}}$. Note that the quotient $Q_{2^{n}}/Z(Q_{2^{n}})$ is dihedral of order $2^{n-1}$. Also note that when $n\geq 4$ the subgroup of $Q_{2^{n}}$ generated by $a^{2}$ and $b$ is isomorphic to $Q_{2^{n-1}}$. In fact, there are exactly 2 subgroups of $Q_{2^{n}}$ that are isomorphic to $Q_{2^{n-1}}$.

\begin{lemma}\label{lem:repthyquatgp}
	Fix an integer $n\geq 3$ and let $Q_{2^{n}}$ be defined as above.
	\begin{itemize}
		\item[(a)] The set of faithful $\chi\in\Irr(Q_{2^{n}})$ is a (nonempty) Galois conjugacy class over $\Q$. Define
		\begin{equation*}
		\gamma_{n}=\sum_{\substack{\chi\in\Irr(Q_{2^{n}})\\\text{faithful}}}\chi.
		\end{equation*}
		\item[(b)] Let $\chi\in\Irr(Q_{2^{n}})$. Then $\chi$ is faithful if and only if $m_{\Q}(\chi)=2$.
		\item[(c)] The quotient $\overline{R}_{\Q}(Q_{2^{n}})/R_{\Q}(Q_{2^{n}})$ is an $\F_{2}$-space of dimension 1, with unique nonzero element (the coset containing) $\gamma_{n}$.
		\item[(d)] Let $\lambda:Z(Q_{2^{n}})\to\Q^{\times}$ be the representation defined by $\lambda(b^{2})=-1$ and let $\Q_{\lambda}$ be a $\Q Z(Q_{2^{n}})$-module affording $\lambda$. Set $\Phi_{Q_{2^{n}}}=\Ind_{Z(Q_{2^{n}})}^{Q_{2^{n}}}(\Q_{\lambda})$. Then $\Phi_{Q_{2^{n}}}$ is a faithful irreducible $\Q Q_{2^{n}}$-module with character $2\gamma_{n}$.
		\item[(e)] If $n\geq 4$ then we have
		\begin{equation*}
		\Ind_{Q_{2^{n-1}}}^{Q_{2^{n}}}(\gamma_{n-1})=\gamma_{n}\qquad\text{and}\qquad\Res_{Q_{2^{n-1}}}^{Q_{2^{n}}}(\gamma_{n})=2\gamma_{n-1},
		\end{equation*}
		where $Q_{2^{n-1}}$ is understood to be either of the 2 subgroups of $Q_{2^{n}}$ isomorphic to $Q_{2^{n-1}}$.
	\end{itemize}
\end{lemma}

\begin{proof}
	The faithful irreducible complex characters of the subgroup $\gp{a}$ induce to faithful irreducible characters of $Q_{2^{n}}$, and any faithful irreducible character of $Q_{2^{n}}$ arises this way. Since Galois conjugation commutes with induction, (a) holds. Let $\chi\in\Irr(Q_{2^{n}})$ and suppose that $\chi$ is not faithful. Then $\chi$ is an irreducible character of the dihedral group $Q_{2^{n}}/Z(Q_{2^{n}})$ so $m_{\Q}(\chi)=1$ by Lemma \ref{lem:schurindicescyclicdihedral}. If $\chi$ is faithful then $m_{\Q}(\chi)=2$: this is shown in \cite[(11.8)]{Feit_1967}. Thus (b) holds and (c) follows. 
	
	From (a) and (b) we know that $2\gamma_{n}$ is the character of a faithful irreducible $\Q Q_{2^{n}}$-module. Therefore to prove (d) it suffices to prove that the character $\Ind_{Z(Q_{2^{n}})}^{Q_{2^{n}}}(\lambda)$ is equal to $2\gamma_{n}$. This follows from a short argument involving Frobenius reciprocity and the fact that the faithful irreducible complex characters of $Q_{2^{n}}$ each have degree 2.
	
	It remains to prove (e). Assume that $n\geq 4$ and denote by $Q_{2^{n-1}}$ a subgroup of $Q_{2^{n}}$ isomorphic to $Q_{2^{n-1}}$. Set $Z=Z(Q_{2^{n}})$. Then $Z=Z(Q_{2^{n-1}})$. We have
	\begin{equation*}
	2\gamma_{n}=\Ind_{Z}^{Q_{2^{n}}}(\lambda)=\Ind_{Q_{2^{n-1}}}^{Q_{2^{n}}}(2\gamma_{n-1})=2\Ind_{Q_{2^{n-1}}}^{Q_{2^{n}}}(\gamma_{n-1}),
	\end{equation*}
	hence $\Ind_{Q_{2^{n-1}}}^{Q_{2^{n}}}(\gamma_{n-1})=\gamma_{n}$. By Mackey's formula we have
	\begin{equation*}
	\Res_{Q_{2^{n-1}}}^{Q_{2^{n}}}(2\gamma_{n})=\Res_{Q_{2^{n-1}}}^{Q_{2^{n}}}(\Ind_{Z}^{Q_{2^{n}}}(\lambda))=2\Ind_{Z}^{Q_{2^{n-1}}}(\lambda)=4\gamma_{n-1},
	\end{equation*}
	so $\Res_{Q_{2^{n-1}}}^{Q_{2^{n}}}(\gamma_{n})=2\gamma_{n-1}$. The proof is complete.
\end{proof}

We remark that $\Phi_{Q_{2^{n}}}=\Ind_{Z(Q_{2^{n}})}^{Q_{2^{n}}}(\Q_{\lambda})$ is the unique faithful irreducible $\Q Q_{2^{n}}$-module up to isomorphism.

Now let $p\in\N$ be prime. A root of unity is a $p'$\textit{-root of unity} if it has order relatively prime to $p$. The set of $p'$-roots of unity in $\C$ forms a subgroup of $\C^{\times}$ and is denoted $\mu_{p'}$.

\begin{lemma}\label{lem:Brauerlem}
	(\cite[Lemma~$1^{\ast}$]{Brauer_1941}) Let $G$ be a finite group and let $K=\Q(\mu_{p'})$, where $p\in\N$ is prime. Then $m_{K}(\chi)=1$ for all $\chi\in\Irr(G)$. In particular, $R_{K}(G)=\overline{R}_{K}(G)$.
\end{lemma}

\section{Biset Functors}

In this section we provide a review of some of the basic notions from the theories of \textit{biset functors}, which was introduced by Bouc in \cite{Bouc_1996}, and of \textit{fibered biset functors}, developed by Boltje and Co\c{s}kun in \cite{Boltje_2018}.

\subsection{Bisets, Burnside Groups, and Biset Categories}\label{subsec:bisets,burnsidegrps}

Let $G$ and $H$ be finite groups. The categories of finite left and right $G$-sets are denoted $\lset{G}$ and $\rset{G}$, respectively, and the category of finite $(G,H)$-bisets is denoted $\biset{G}{H}$. The \textit{Burnside groups} $B(G)$ and $B(G,H)$ are defined to be the Grothendieck groups, with respect to disjoint unions, of $\lset{G}$ and $\biset{G}{H}$. Any right $H$-set $X$ may be viewed as a left $H$-set via the rule $h\cdot x=xh^{-1}$ for all $h\in H$ and $x\in X$, and this convention induces category isomorphisms $\rset{H}\iso\lset{H}$ and $\biset{G}{H}\iso\lset{G\times H}$. In particular, $B(G,H)\iso B(G\times H)$.

Let $A$ be an abelian (potentially infinite) group. An $A$\textit{-fibered} $G$\textit{-set} is a (left) $G$-set equipped with a free, two-sided action of $A$ that has finitely many $A$-orbits and commutes with the action of $G$. Any $A$-fibered $G$-set may be viewed as an $A\times G$-set in a natural way. The category of $A$-fibered $G$-sets is denoted $\lset{G}^{A}$. One also has the notion of $A$\textit{-fibered} $(G,H)$\textit{-bisets} and the category $\biset{G}{H}^{A}$. Moreover, there is a category isomorphism $\biset{G}{H}^{A}\iso\lset{G\times H}^{A}$.

The \textit{(}$A$\textit{-fibered) Burnside group} $B^{A}(G)$ of $A$-fibered $G$-sets, which was introduced by Dress in \cite{Dress_1971}, is defined as the Grothendieck group, with respect to disjoint unions, of the category $\lset{G}^{A}$. Likewise the Burnside group $B^{A}(G,H)$ is the Grothendieck group of the category $\biset{G}{H}^{A}$. Note that $B(G)\iso B^{\set{1}}(G)$ and that $B^{A}(G,H)\iso B^{A}(G\times H)$ --- we identify the groups in these pairs. The construction $X\mapsto A\times X$ produces a split injective homomorphism $B(G)\to B^{A}(G)$ with left inverse $B^{A}(G)\to B(G)$ induced by $Y\mapsto Y/A$.

The isomorphism classes of transitive $A$-fibered $G$-sets form a $\Z$-basis of $B^{A}(G)$ which we call the \textit{standard basis}. There is a bijection between the set of isomorphism classes $[X]$ of transitive $A$-fibered $G$-sets and the set of $G$-conjugacy classes of pairs $(U,\phi)$ where $U\subgp G$ and $\phi:U\to A$ is a group homomorphism: to such a class $[X]$ one associates the $G$-conjugacy class of the \textit{stabilizing pair} $(U_{x},\phi_{x})$ where $x\in X$, $U_{x}$ is the stabilizer in $G$ of the $A$-orbit of $x$, and $\phi_{x}:U_{x}\to A$ is the (well-defined) group homomorphism given by the equation $gx=\phi_{x}(g)x$ for each $g\in U_{x}$. If $X$ is a transitive $A$-fibered $G$-set whose isomorphism class $[X]$ corresponds to the conjugacy class of the pair $(U,\phi)$ then we write
\begin{equation*}
[X]=\left[\frac{G}{U,\phi}\right].
\end{equation*} 
Similar statements hold for $A$-fibered $(G,H)$-bisets, and in particular for $A$-fibered $G\times H$-sets. 

Let $G$, $H$, and $K$ be groups. If $X$ is a $(G,H)$-biset and $Y$ is an $(H,K)$-biset then the Cartesian product $X\times Y$ is a left $H$-set via $h\cdot(x,y)=(xh^{-1},hy)$ for all $h\in H$, $x\in X$, and $y\in Y$. The set of $H$-orbits is a $(G,K)$-biset denoted $X\times_{H}Y$. The operation $\times_{H}$ is associative up to a canonical isomorphism and extends uniquely to a bilinear map $\times_{H}:B(G,H)\times B(H,K)\to B(G,K)$. If in addition $X$ and $Y$ are $A$-fibered then in a similar manner one may construct an $A$-fibered $(G,K)$-biset denoted $X\tensor_{AH}Y$.\footnote{Some extra care is needed here to ensure that the resulting biset has a free $A$-action. The precise definition is given in \cite[Section~2.1]{Boltje_2018}.} Again the operation $\tensor_{AH}$ is associative up to a canonical isomorphism and extends uniquely to a bilinear map $\tensor_{AH}:B^{A}(G,H)\times B^{A}(H,K)\to B^{A}(G,K)$.

Bouc defines the \textit{biset category} $\CC$ as follows (see \cite[Section~3.1]{Bouc_2010}): the objects of $\CC$ are all finite groups. If $G$ and $H$ are finite groups then we set $\Hom_{\CC}(H,G)=B(G,H)$. If also $K$ is a finite group then the composition $v\circ u$ of $v\in\Hom_{\CC}(H,G)$ and $u\in\Hom_{\CC}(K,H)$ is $v\times_{H}u$. The isomorphism class of the $(G,G)$-biset $G$ --- a biset under left and right multiplication --- is the identity of $G$. We note that the biset category $\CC$ is preadditive.

In \cite{Boltje_2018}, Boltje and Co\c{s}kun define the $A$\textit{-fibered biset category} $\CC^{A}$ to be the preadditive category whose objects are all finite groups and whose morphism sets $\Hom_{\CC^{A}}(H,G)=B^{A}(G,H)$ are the $A$-fibered Burnside groups. If $G$, $H$, and $K$ are finite groups then the composition $v\circ u$ of $v\in\Hom_{\CC^{A}}(H,G)$ and $u\in\Hom_{\CC^{A}}(K,H)$ is $v\tensor_{AH}u$. The identity of $G$ is the isomorphism class of the $A$-fibered $(G,G)$-biset $A\times G$.

The injective group homomorphisms $B(G,H)\to B^{A}(G,H)$ described above induce a faithful additive functor $\CC\to\CC^{A}$ that is the identity on objects. We therefore view $\CC$ as a subcategory of $\CC^{A}$ without further comment.

Let $p$ be a prime number. The $p$\textit{-biset category} $\CC_{p}$ is the full subcategory of the biset category $\CC$ whose objects are the finite $p$-groups. Similarly, one has the $A$\textit{-fibered} $p$\textit{-biset category} $\CC_{p}^{A}$.

We recall the definitions of some important bisets. Let $f:H\to G$ be a homomorphism of finite groups. Set
\begin{equation*}
{}_{f}\Delta(H)=\set{(f(h),h):h\in H}\qquad\text{and}\qquad\Delta_{f}(H)=\set{(h,f(h)):h\in H}.
\end{equation*}
If $H$ is a subgroup of $G$ and $f$ is the inclusion map we write $\Delta(H)$ for ${}_{f}\Delta(H)$ or $\Delta_{f}(H)$. If $f$ is an isomorphism set
\begin{equation*}
\Iso(f)=\left[\frac{G\times H}{{}_{f}\Delta(H),1}\right].
\end{equation*}
Notice that $\Iso(f)$ is an isomorphism between $H$ and $G$ and that $\Iso(\id_{G})$ is the identity of $G$ in $\CC^{A}$. If $H\subgp G$ set
\begin{equation*}
\Ind_{H}^{G}=\left[\frac{G\times H}{\Delta(H),1}\right]\qquad\text{and}\qquad\Res_{H}^{G}=\left[\frac{H\times G}{\Delta(H),1}\right].
\end{equation*}
These morphisms are called \textit{induction} and \textit{restriction}, respectively. Now let $N\nor G$ and let $\pi:G\to G/N$ denote the canonical projection. The morphisms
\begin{equation*}
\Inf_{G/N}^{G}=\left[\frac{G\times G/N}{\Delta_{\pi}(G),1}\right]\qquad\text{and}\qquad\Def_{G/N}^{G}=\left[\frac{G/N\times G}{{}_{\pi}\Delta(G),1}\right]
\end{equation*}
are called \textit{inflation} and \textit{deflation}, respectively. If $(H,N)$ is a \textit{section} of $G$ --- that is, if $N\nor H\subgp G$ --- then set
\begin{equation*}
\Indinf_{H/N}^{G}=\Ind_{H}^{G}\circ\Inf_{H/N}^{H}\qquad\text{and}\qquad\Defres_{H/N}^{G}=\Def_{H/N}^{H}\circ\Res_{H}^{G}.
\end{equation*}
Finally, let $\phi:G\to A$ be a group homomorphism. Precomposition with the canonical group isomorphism $\Delta(G)\isoto G$ produces a group homomorphism $\Delta(G)\to A$ which we denote, abusively, by $\phi$. Set
\begin{equation*}
\Mult(\phi)=\left[\frac{G\times G}{\Delta(G),\phi}\right].
\end{equation*}
Note that $\Mult(\phi)$, which we call \textit{multiplication by }$\phi$, is an endomorphism of $G$ in $\CC^{A}$ (but is not an endomorphism in $\CC$, in general). We remark that multiplication by $\phi$ is called a ``twist with $\phi$'' by Boltje and Y{\i}lmaz in \cite[Section~3.7(d)]{Boltje_2019}.

\subsection{Biset Functors}

We continue with the notation of the previous subsection: $G$ and $H$ stand for finite groups and $A$ is a fixed abelian group. Let $\DD$ be a preadditive subcategory of the $A$-fibered biset category $\CC^{A}$. An $A$\textit{-fibered biset functor defined on }$\DD$ is an additive functor from $\DD$ to $\lMod{\Z}$, the category of left $\Z$-modules. If $\DD\subseteq\CC$ then an $A$-fibered biset functor defined on $\DD$ is simply called a \textit{biset functor}, and if $\DD=\CC_{p}$ for some prime $p$ --- that is, if $\DD$ is the full subcategory of $\CC$ whose objects are the finite $p$-groups --- then a biset functor defined on $\DD$ is called a $p$\textit{-biset functor}.

When $\DD$ is essentially small the $A$-fibered biset functors defined on $\DD$ are the objects of a category $\FF_{\DD}^{A}$ whose morphisms are natural transformations. Since $\lMod{\Z}$ is abelian, so is $\FF_{\DD}^{A}$: kernels and cokernels are constructed ``pointwise'' at each group $G\in\DD$. Moreover, a morphism $\eta$ is monic/epic if and only if $\eta_{G}$ is injective/surjective for all $G\in\DD$. For simplicity we write $\FF^{A}$ in place of $\FF_{\CC^{A}}^{A}$ and we write $\FF$ in place of $\FF_{\CC}^{A}$. In addition, if $\DD=\CC_{p}$ for some prime $p$ then $\FF_{p}$ denotes the category of $p$-biset functors. 

When $F\in\FF_{\DD}^{A}$, $\theta\in F(H)$ for some $H\in\DD$, and $u\in\Hom_{\DD}(H,G)$ we sometimes write $u(\theta)$ instead of $F(u)(\theta)$. Thus we may write $\Ind_{H}^{G}(\theta)$ rather than $F(\Ind_{H}^{G})(\theta)$, for example.

A homomorphism $A'\to A$ of abelian groups induces homomorphisms $B^{A'}(G,H)\to B^{A}(G,H)$ for all finite groups $G$ and $H$. These maps in turn induce an additive functor $\CC^{A'}\to\CC^{A}$, and restriction along this functor yields an additive functor $\FF^{A}\to\FF^{A'}$. In particular, if $A'$ is a subgroup of $A$ then any $A$-fibered biset functor restricts to an $A'$-fibered biset functor. The inclusion $\Tor(A)\subgp A$ induces isomorphisms $B^{\Tor(A)}(G,H)\isoto B^{A}(G,H)$, so the resulting functors $\CC^{\Tor(A)}\to\CC^{A}$ and $\FF^{A}\to\FF^{\Tor(A)}$ are category isomorphisms. In particular, no structure is lost when restricting an $A$-fibered biset functor to a $\Tor(A)$-fibered biset functor.

If $F'$ and $F$ are objects in $\FF_{\DD}^{A}$ then we say that $F'$ is a \textit{subfunctor} of $F$ and we write $F'\subseteq F$ if $F'(G)$ is a subgroup of $F(G)$ for all $G\in\DD$ and if the inclusion maps $\iota_{G}:F'(G)\to F(G)$ define a natural transformation $\iota:F'\to F$. Note that a subfunctor $F'$ of $F$ is uniquely determined by the data of a subgroup $F'(G)$ of $F(G)$, for each $G\in\DD$, such that $F(u)$ restricts to a homomorphism $F'(H)\to F'(G)$ for all $u\in\Hom_{\DD}(H,G)$. If $F'$ and $F''$ are subfunctors of $F\in\FF_{\DD}^{A}$ then $F'=F''$ if and only if $F'(G)=F''(G)$ for all groups $G\in\DD$. Given a set of subfunctors $\set{F_{j}}$ of an $A$-fibered biset functor $F$ defined on $\DD$ the \textit{intersection} $\cap F_{j}$ is the subfunctor of $F$ defined by $(\cap F_{j})(G)=\cap F_{j}(G)$ for all $G\in\DD$. One also has the notion of a subfunctor generated by elements: let $S$ be a set of objects of $\DD$ and for each $H\in S$ let $S_{H}$ be a subset of $F(H)$. The subfunctor of $F$ \textit{generated by} $\sqcup_{H\in S}S_{H}$ is the intersection of all subfunctors $F'$ of $F$ such that $F'(H)\supseteq S_{H}$ for all $H\in S$. When $\sqcup_{H\in S}S_{H}$ consists of a single element $\gamma$ then the subfunctor generated by $\sqcup_{H\in S}S_{H}$ will instead be called the subfunctor generated by $\gamma$. Note that if $F'$ is the subfunctor of $F$ generated by $\gamma\in F(H)$ then we have
\begin{equation*}
F'(G)=\Hom_{\DD}(H,G)(\gamma)=\set{u(\gamma):u\in\Hom_{\DD}(H,G)}
\end{equation*}
for all $G\in\DD$.

\subsection{Examples}\label{subsec:examples}
	
\begin{example} (The virtual character ring)
	Let $K$ be a field of characteristic 0 and let $G$ be a finite group. As in \cite[Section~11B]{Boltje_2018} the construction $R_{K}:G\mapsto R_{K}(G)$ may be extended to a $K^{\times}$-fibered biset functor defined on the full $K^{\times}$-fibered biset category $\CC^{K^{\times}}$ by setting $R_{K}\left(\left[\frac{G\times H}{U,\phi}\right]\right)$, for any finite groups $G$, $H$ and any standard basis element $\left[\frac{G\times H}{U,\phi}\right]\in B^{K^{\times}}(G,H)$, equal to the map
	\begin{align*}
	R_{K}\left(\left[\frac{G\times H}{U,\phi}\right]\right):R_{K}(H)&\to R_{K}(G)\\
	[V]&\mapsto[\Ind_{U}^{G\times H}(K_{\phi})\tensor_{KH}V]
	\end{align*}
	where $V$ is any $KH$-module. The images under $R_{K}$ of the morphisms $\Ind_{H}^{G}$, $\Res_{H}^{G}$, $\Inf_{G/N}^{G}$, etc. which were introduced in Subsection \ref{subsec:bisets,burnsidegrps} coincide with the familiar operations from representation theory. For example, if $H\subgp G$ then $R_{K}(\Ind_{H}^{G})$ is the map induced by extension of scalars $KG\tensor_{KH}\cdot$, which justifies the abuse of notation $\Ind_{H}^{G}=R_{K}(\Ind_{H}^{G})$. Also, if $\phi:G\to K^{\times}$ is a homomorphism and $\chi$ is a $K$-character of $G$ then $R_{K}(\Mult(\phi))(\chi)=\phi\chi$.
	
	Let $L$ be an extension of $K$. Since $K^{\times}\subgp L^{\times}$ we may view $R_{L}$ as a $K^{\times}$-fibered biset functor. The maps $\iota_{G}:R_{K}(G)\into R_{L}(G)$ induced by extension of scalars from $K$ to $L$ are the components of an injective (monic) natural transformation $\iota:R_{K}\into R_{L}$. We identify $R_{K}$ with its image under $\iota$ so that $R_{K}$ is a $K^{\times}$-fibered subfunctor of $R_{L}$. 
	
	Now let $L$ be an algebraically closed extension of $K$. Recall that $\overline{R}_{K}(G)$ denotes the subring of $K$-valued virtual characters in $R_{L}(G)$. The character formula in the first part of \cite[Lemma~7.1.3]{Bouc_2010} can be used to show that $\overline{R}_{K}$ is a $K^{\times}$-fibered subfunctor of $R_{L}$. Since $R_{K}(G)\subseteq\overline{R}_{K}(G)$ for all finite groups $G$ it follows that $R_{K}$ is a subfunctor of $\overline{R}_{K}$. In particular one may form the quotient $\overline{R}_{K}/R_{K}\in\FF^{K^{\times}}$, whose evaluation at a group $G$ is given by
	\begin{equation*}
	(\overline{R}_{K}/R_{K})(G)=\overline{R}_{K}(G)/R_{K}(G).
	\end{equation*}
\end{example}

\begin{example}\label{ex:trivsourcering} (The trivial source ring)
	Let $(\K,\OO,k)$ be a $p$-modular system (see Section \ref{sec:intro}) and assume that $k$ is algebraically closed. Then $\OO$ contains all $p'$-roots of unity. Fix an embedding $\Q(\mu_{p'})\into\K$, where as before $\mu_{p'}$ denotes the group of complex roots of unity having order prime to $p$. Then $\mu_{p'}\subgp\OO^{\times}$ and $\mu_{p'}$ maps bijectively onto the group of roots of unity in $k$ via the canonical projection $\OO\onto k$. We identify $\mu_{p'}$ with its image under this projection.
	
	Recall that an $\OO G$-module is a \textit{trivial source} (or $p$\textit{-permutation})\textit{ module} if it is isomorphic to a direct summand of a permutation $\OO G$-module. The Grothendieck ring of the category of trivial source $\OO G$-modules, with respect to split short exact sequences, is denoted $T_{\OO}(G)$ and is called the \textit{trivial source ring} (see \cite[Section~5.5]{Benson_1991_vol1}). Following the procedure of the previous example the construction $T_{\OO}:G\mapsto T_{\OO}(G)$ may be given the structure of a $\mu_{p'}$-fibered biset functor. 
	
	For each finite group $G$ there is a homomorphism $\kappa_{G}:T_{\OO}(G)\to R_{\K}(G)$ induced by extension of scalars from $\OO$ to $\K$. Since $\mu_{p'}\subgp\K^{\times}$ we may regard $R_{\K}$ as a $\mu_{p'}$-fibered biset functor. Then $\kappa:T_{\OO}\to R_{\K}$ is a morphism in $\FF^{\mu_{p'}}$. Now the image $\im(\kappa)$ of $\kappa$ is a subfunctor of $R_{\K}$, and by \cite[Theorem~1]{Dress_1975} the evaluation of $\im(\kappa)$ at a group $G$ is equal to
	\begin{equation*}
	\im(\kappa_{G})=\gp{[\Ind_{H}^{G}(\K_{\phi})]:H\subgp G,\phi:H\to\mu_{p'}}.
	\end{equation*}
	Observe that $\im(\kappa_{G})$ is a subgroup (in fact, a subring) of $R_{\Q(\mu_{p'})}(G)$ for any group $G$. Thus $\im(\kappa)$ is a $\mu_{p'}$-fibered subfunctor of $R_{\Q(\mu_{p'})}$. In Section \ref{sec:subfunctorsofRKmodimkappa} we determine the subfunctor structure of the quotient $R_{\Q(\mu_{p'})}/\im(\kappa)$.
\end{example}

\subsection{Induction of Fibered Biset Functors}\label{subsec:restrandind}

Fix an abelian group $A$ and an essentially small preadditive subcategory $\DD$ of $\CC^{A}$. Restriction of $A$-fibered biset functors from $\CC^{A}$ to $\DD$ may be realized as an additive and exact functor $\RES:\FF^{A}\to\FF_{\DD}^{A}$, called \textit{restriction}. The purpose of this subsection is to introduce the left adjoint $\IND:\FF_{\DD}^{A}\to\FF^{A}$, called \textit{induction}.

Let $F:\DD\to\lMod{\Z}$ be an $A$-fibered biset functor defined on $\DD$. Define the additive functor $\IND(F):\CC^{A}\to\lMod{\Z}$ as follows: for any $G\in\CC^{A}$,
\begin{equation*}
\IND(F)(G)=\left(\underset{P\in \SS}{\Dsum}\Hom_{\CC^{A}}(P,G)\tensor_{\Z}F(P)\right)/\RR_{G}
\end{equation*}
where $\SS$ is a fixed skeleton of $\DD$ and $\RR_{G}$ is the subgroup of the direct sum generated by all elements of the form
\begin{equation*}
(v\circ w)\tensor\theta-v\tensor w(\theta),
\end{equation*}
where $w\in\Hom_{\DD}(P,Q)$ for $P,Q\in\SS$, $\theta\in F(P)$, and $v\in\Hom_{\CC^{A}}(Q,G)$. If $u\in\Hom_{\CC^{A}}(H,G)$ then $\IND(F)(u)$ is the group homomorphism satisfying
\begin{equation*}
\IND(F)(u)(v\tensor\theta)=(u\circ v)\tensor\theta.
\end{equation*}
Here $v\tensor\theta$ is a typical ``simple tensor'' in $\IND(F)(H)$: $v\in\Hom_{\CC^{A}}(P,H)$, $P\in\SS$, and $\theta\in F(P)$. 

The construction $\IND:F\mapsto\IND(F)$ extends to a functor from $\FF_{\DD}^{A}$ to $\FF^{A}$: let $F$ and $F'$ be objects of $\FF_{\DD}^{A}$ and let $\eta\in\Hom_{\FF_{\DD}^{A}}(F,F')$ be a natural transformation. For each $H\in\CC^{A}$ set $\IND(\eta)_{H}:\IND(F)(H)\to\IND(F')(H)$ equal to the group homomorphism defined by
\begin{equation*}
\IND(\eta)_{H}(v\tensor\theta)=v\tensor\eta_{P}(\theta),
\end{equation*}
where $v\tensor\theta$ is a simple tensor of $\IND(F)(H)$, as above. Then $\IND(\eta):\IND(F)\to\IND(F')$ is a natural transformation (i.e., a morphism in $\FF^{A}$). With these definitions in hand, checking that $\IND$ is a functor from $\FF_{\DD}^{A}$ to $\FF^{A}$ is straightforward.

\begin{proposition}\label{prop:resandindadjunct}
	Let $A$ be an abelian group and let $\DD$ be an essentially small preadditive subcategory of $\CC^{A}$. Then induction $\IND:\FF_{\DD}^{A}\to\FF^{A}$ is left adjoint to restriction $\RES:\FF^{A}\to\FF_{\DD}^{A}$.
\end{proposition}

\begin{proof}
	Fix a skeleton $\SS$ of $\DD$, let $F\in\FF_{\DD}^{A}$, and let $F'\in\FF^{A}$. Define a natural bijection
	\begin{equation*}
	\psi=\psi_{F,F'}:\Hom_{\FF_{\DD}^{A}}(F,\RES(F'))\to\Hom_{\FF^{A}}(\IND(F),F')
	\end{equation*}
	as follows: if $\zeta:F\to\RES(F')$ is a natural transformation then let $\psi(\zeta):\IND(F)\to F'$ be the natural transformation that satisfies
	\begin{equation*}
	\psi(\zeta)_{H}(v\tensor\theta)=F'(v)(\zeta_{P}(\theta)).
	\end{equation*}
	In the above, $H\in\CC^{A}$, $P\in\SS$, $v\in\Hom_{\CC^{A}}(P,H)$, and $\theta\in F(P)$. 
\end{proof}

We remark that $\IND(F)$ can alternatively be defined as a \textit{coend} of the functor $S:\op{\DD}\times\DD\to\FF^{A}$ that maps an object $(P,Q)$ to the $A$-fibered biset functor $\Hom_{\CC^{A}}(P,\cdot)\tensor_{\Z}F(Q)$. Using \textit{ends} one may also define the right adjoint of $\RES$. We refer to \cite[Section~3.3]{Bouc_2010} for more information.

\section{Rational $p$-Biset Functors}\label{sec:rationalpbisetfunctors}

Fix a prime number $p$. Throughout this section $R_{\Q}$, $\overline{R}_{\Q}$, and $\overline{R}_{\Q}/R_{\Q}$ are regarded only as $p$-biset functors, i.e., as functors defined on $\CC_{p}$, the full subcategory of the biset category $\CC$ with objects all finite $p$-groups. We provide a review of Bouc's theory of \textit{rational }$p$\textit{-biset functors}, focusing only on the results that will be needed later (the reader should consult \cite[Chapters 9 and 10]{Bouc_2010} for a more in-depth discussion of this material). We then apply this theory to determine the subfunctor structure of the $p$-biset functor $\overline{R}_{\Q}/R_{\Q}$. Many of the results on rational representations of $p$-groups in this section go back to Roquette; see \cite{Roquette_1958}.

Let $P$ be a $p$-group. If $P$ does not contain any normal subgroup isomorphic to $C_{p}\times C_{p}$ then $P$ is said to have \textit{normal }$p$\textit{-rank }1. If $p$ is odd then the $p$-groups of normal $p$-rank 1 are the cyclic groups. The $2$-groups of normal $2$-rank 1 are the cyclic groups, generalized quaternion groups, dihedral groups of order at least $16$, and the semidihedral groups. If $P$ has normal $p$-rank 1 then there is a unique (up to isomorphism) faithful irreducible $\Q P$-module, denoted $\Phi_{P}$.

If $P$ is a $p$-group and $S\subgp P$ is such that $N_{P}(S)/S$ has normal $p$-rank 1 then set
\begin{equation*}
V(S)=\Indinf_{N_{P}(S)/S}^{P}(\Phi_{N_{P}(S)/S}).
\end{equation*}
The functor $\Indinf_{N_{P}(S)/S}^{P}:\lmod{\Q N_{P}(S)/S}\to\lmod{\Q P}$ induces an injective $\Q$-algebra homomorphism
\begin{equation*}
\End_{\Q N_{P}(S)/S}(\Phi_{N_{P}(S)/S})\into\End_{\Q P}(V(S)).
\end{equation*}
Following Bouc, we say that $S$ is a \textit{genetic subgroup} of $P$ if this map is an isomorphism of $\Q$-algebras. A group-theoretic characterization of genetic subgroups is given in \cite[Theorem~9.5.6]{Bouc_2010}.

If $S$ is a genetic subgroup of $P$ then $\End_{\Q P}(V(S))$ is a division $\Q$-algebra. It follows that $V(S)$ is an indecomposable, hence irreducible $\Q P$-module. In this way, genetic subgroups of $P$ give rise to irreducible $\Q P$-modules. It turns out (\cite[Corollary~9.4.5]{Bouc_2010}) that any irreducible $\Q P$-module $V$ arises in this fashion, up to an isomorphism: there exists a genetic subgroup $S$ of $P$ such that $V\iso V(S)$ (we remark that Lemma \ref{lem:roquette} follows from these observations). Define an equivalence relation $\sim$ on the set of genetic subgroups of $P$ by setting $S\sim T$ if and only if there is a $\Q P$-isomorphism $V(S)\iso V(T)$. Bouc calls a set of representatives for the equivalence classes of $\sim$ a \textit{genetic basis} of $P$. Of course, if $\mathcal{G}$ is a genetic basis of $P$ then $|\mathcal{G}|=|\Irr_{\Q}(P)|$. We remark that \cite[Theorem~9.6.1]{Bouc_2010} shows that the group structure of $P$ determines precisely when two genetic subgroups $S$ and $T$ of $P$ satisfy $V(S)\iso V(T)$. In particular, the relation $\sim$ can be redefined solely in terms of the group structure of $P$. Furthermore, it can be shown that if $S$ and $T$ are genetic subgroups of $P$ satisfying $S\sim T$ then the groups $N_{P}(S)/S$ and $N_{P}(T)/T$ are isomorphic.

Now let $F\in\FF_{p}$ be a $p$-biset functor and let $P$ be a $p$-group. Following Bouc (see \cite[Definition~6.3.1 and Lemma~6.3.2]{Bouc_2010}), set
\begin{equation*}
\partial F(P)=\underset{\set{1}\neq N\nor P}{\bigcap}\ker(\Def_{P/N}^{P}:F(P)\to F(P/N)),
\end{equation*}
the subgroup of \textit{faithful elements} of $F(P)$.

\begin{proposition}\label{prop:partialRbar}
	Let $P$ be a $p$-group and let $\chi_{1},\chi_{2},\ldots,\chi_{r}$ denote representatives for the Galois conjugacy classes of $\Irr(P)$ over $\Q$. Recall that the Galois class sums $\overline{\chi_{i}}$ form a $\Z$-basis of $\overline{R}_{\Q}(P)$. The subgroup $\partial\overline{R}_{\Q}(P)$ of faithful elements of $\overline{R}_{\Q}(P)$ is generated by the $\overline{\chi_{i}}$ for which $\chi_{i}$ is faithful. In particular, $\partial\overline{R}_{\Q}(P)$ is a free $\Z$-module and if $P$ has normal $p$-rank 1 then $\partial\overline{R}_{\Q}(P)$ has $\Z$-rank 1.
\end{proposition}

\begin{proof}
	 Notice that if $\chi,\psi\in\Irr(P)$ are Galois conjugate over $\Q$ then $\ker(\chi)=\ker(\psi)$. If $N\nor P$ recall that $\Irr(P/N)$ can be identified with the set $\set{\chi\in\Irr(P):N\subgp\ker(\chi)}$. Under this identification, a Galois conjugacy class of irreducible characters of $P/N$ is equal to a Galois conjugacy class of irreducible characters of $P$. For any Galois class sum $\overline{\chi_{i}}$ we have
	 \begin{equation*}
	 \Def_{P/N}^{P}(\overline{\chi_{i}})=\begin{cases}
	 	\overline{\chi_{i}}	&\text{if }N\subgp\ker(\chi_{i})\\
	 	0					&\text{else.}
	 \end{cases}
	 \end{equation*}
	 It follows from this formula that $\partial\overline{R}_{\Q}(P)$ is generated by the $\overline{\chi_{i}}$ for which $\chi_{i}$ is faithful. In particular, $\partial\overline{R}_{\Q}(P)$ is a free $\Z$-module.
	 
	 Now let $P$ be a $p$-group of normal $p$-rank 1. Recall that $\Phi_{P}$ is the unique faithful irreducible $\Q P$-module, up to isomorphism. The character of any irreducible constituent of $\C\tensor_{\Q}\Phi_{P}$ is faithful, so the set of faithful $\chi\in\Irr(P)$ is nonempty. If we show that this set forms a full Galois conjugacy class over $\Q$ then the proof will be complete. So let $\chi,\psi\in\Irr(P)$ be faithful. The characters $m_{\Q}(\chi)\overline{\chi}$ and $m_{\Q}(\psi)\overline{\psi}$ are afforded by faithful irreducible $\Q P$-modules, thus both are afforded by $\Phi_{P}$. But if $m_{\Q}(\chi)\overline{\chi}=m_{\Q}(\psi)\overline{\psi}$ then $\chi$ and $\psi$ must be Galois conjugate over $\Q$.
\end{proof}

\begin{theorem}\label{thm:JGinjec}
	(\cite[Theorem~10.1.1]{Bouc_2010}) Let $P$ be a $p$-group, $\mathcal{G}$ a genetic basis of $P$, and let $F$ be a $p$-biset functor. Then the homomorphism
	\begin{equation*}
	\mathfrak{I}_{\mathcal{G}}=\underset{S\in\mathcal{G}}{\Dsum}\Indinf_{N_{P}(S)/S}^{P}:\underset{S\in\mathcal{G}}{\Dsum}\partial F(N_{P}(S)/S)\to F(P)
	\end{equation*}
	is split injective.
\end{theorem}

\begin{definition}
	(Bouc) A $p$-biset functor $F$ is \textit{rational} if for any $p$-group $P$ and any genetic basis $\mathcal{G}$ of $P$ the map $\mathfrak{I}_{\mathcal{G}}$ is an isomorphism.
\end{definition}

It is immediate that the $p$-biset functor $R_{\Q}$ is rational and in fact $R_{\Q}$ is the prototypical example of a rational $p$-biset functor. As Barker observed in \cite[Example~3.A]{Barker_2008} the same is true if $\Q$ is replaced by any field of characteristic 0. In particular, $R_{\C}$ is a rational $p$-biset functor. Since by \cite[Theorem~10.1.5]{Bouc_2010} the class of rational $p$-biset functors is closed under taking subfunctors and quotients it follows that $\overline{R}_{\Q}$ and $\overline{R}_{\Q}/R_{\Q}$ are both rational.\footnote{The fact that $\overline{R}_{\Q}$ is a rational $p$-biset functor also follows directly from the definitions and the well-known interpretation of Schur indices as indices of certain endomorphism algebras; see \cite[Section~12.2]{Serre_1977}.}

Recall that, by Lemma \ref{lem:roquette}, the $p$-biset functor $\overline{R}_{\Q}/R_{\Q}$ is trivial if $p\neq 2$. We turn our attention now to the nontrivial case $p=2$. By Lemma \ref{lem:schurindicescyclicdihedral} if $P$ is a cyclic, dihedral, or semidihedral group then $(\overline{R}_{\Q}/R_{\Q})(P)=\set{0}$. Since $\partial(\overline{R}_{\Q}/R_{\Q})(Q_{2^{n}})=(\overline{R}_{\Q}/R_{\Q})(Q_{2^{n}})$ for any $n\geq 3$, if $P$ is a $2$-group with genetic basis $\mathcal{G}$ then the map $\mathfrak{I}_{\mathcal{G}}$ can alternatively be defined:
\begin{equation*}
\mathfrak{I}_{\mathcal{G}}=\underset{S\in\mathcal{Q}}{\Dsum}\Indinf_{N_{P}(S)/S}^{P}:\underset{S\in\mathcal{Q}}{\Dsum} (\overline{R}_{\Q}/R_{\Q})(N_{P}(S)/S)\isoto(\overline{R}_{\Q}/R_{\Q})(P),
\end{equation*}
where 
\begin{equation*}
\mathcal{Q}=\set{S\in\mathcal{G}:N_{P}(S)/S\iso Q_{2^{n}}\text{, some }n\geq 3}.
\end{equation*} 
In this case the inverse $\mathfrak{D}_{\mathcal{G}}$ of $\mathfrak{I}_{\mathcal{G}}$ can be defined as below (see Corollary~6.4.5, Lemma~9.5.2, and Theorem~9.6.1 of \cite{Bouc_2010}):
\begin{equation*}
\mathfrak{D}_{\mathcal{G}}=\underset{S\in\mathcal{Q}}{\Dsum}\Defres_{N_{P}(S)/S}^{P}:(\overline{R}_{\Q}/R_{\Q})(P)\isoto\underset{S\in\mathcal{Q}}{\Dsum}(\overline{R}_{\Q}/R_{\Q})(N_{P}(S)/S).
\end{equation*}

In the following theorem we make use of the notation set up in Lemma \ref{lem:repthyquatgp}. In particular, recall that
\begin{equation*}
\gamma_{n}=\sum_{\substack{\chi\in\Irr(Q_{2^{n}})\\\text{faithful}}}\chi
\end{equation*}
is a representative for the unique nonzero coset in $(\overline{R}_{\Q}/R_{\Q})(Q_{2^{n}})$. For ease, we write $\gamma_{n}$ for this coset.

\begin{theorem}\label{thm:subfunctorsofRbarQ/RQ}
	For each integer $n\geq 3$ let $F_{n}$ denote the subfunctor of the $2$-biset functor $\overline{R}_{\Q}/R_{\Q}$ generated by $\gamma_{n}$. Then the subfunctors $F_{n}$ are rational $2$-biset functors and are precisely the nonzero subfunctors of $\overline{R}_{\Q}/R_{\Q}$. We have
	\begin{equation*}
	\overline{R}_{\Q}/R_{\Q}=F_{3}\supsetneq F_{4}\supsetneq F_{5}\supsetneq\cdots.
	\end{equation*}
	In particular, $\overline{R}_{\Q}/R_{\Q}$ is generated by $\gamma_{3}$.
\end{theorem}

\begin{proof}
	Since subfunctors of rational $p$-biset functors are rational, each $F_{n}$, $n\geq 3$, is a rational $2$-biset functor. Let $n\geq 4$ and view $Q_{2^{n-1}}$ as a subgroup of $Q_{2^{n}}$. Recall from Lemma \ref{lem:repthyquatgp} that $\Ind_{Q_{2^{n-1}}}^{Q_{2^{n}}}(\gamma_{n-1})=\gamma_{n}$. For any $2$-group $P$ we have 
	\begin{align*}
	F_{n}(P)&=\Hom_{\CC_{2}}(Q_{2^{n}},P)(\gamma_{n})=\Hom_{\CC_{2}}(Q_{2^{n}},P)(\Ind_{Q_{2^{n-1}}}^{Q_{2^{n}}}(\gamma_{n-1}))\\
	&\subseteq\Hom_{\CC_{2}}(Q_{2^{n-1}},P)(\gamma_{n-1})=F_{n-1}(P).
	\end{align*}
	Thus $F_{n}$ is a subfunctor of $F_{n-1}$.
	
	Keeping $n\geq 4$ we now show that $F_{n}$ is a proper subfunctor of $F_{n-1}$. Since $F_{n-1}(Q_{2^{n-1}})\neq \set{0}$ it suffices to show that $F_{n}(Q_{2^{n-1}})=\set{0}$. Suppose that this is not the case. Then there must exist a transitive $(Q_{2^{n-1}},Q_{2^{n}})$-biset $X$ such that $[X]$ does not annihilate $\gamma_{n}$, i.e., such that $\left[(\overline{R}_{\Q}/R_{\Q})([X])\right](\gamma_{n})\neq 0$. By \cite[Lemma~2.3.26]{Bouc_2010} there exists a section $(P_{1},K_{1})$ of $Q_{2^{n-1}}$, a section $(P_{2},K_{2})$ of $Q_{2^{n}}$, and a group isomorphism $f:P_{2}/K_{2}\to P_{1}/K_{1}$ such that
	\begin{equation*}
	[X]=\Indinf_{P_{1}/K_{1}}^{Q_{2^{n-1}}}\circ\Iso(f)\circ\Defres_{P_{2}/K_{2}}^{Q_{2^{n}}}.
	\end{equation*} 
	Since $[X]$ does not annihilate $\gamma_{n}$, neither does $\Defres_{P_{2}/K_{2}}^{Q_{2^{n}}}$. Now, if $P_{2}$ is a proper subgroup of $Q_{2^{n}}$ then $\Res_{P_{2}}^{Q_{2^{n}}}(\gamma_{n})=0$ --- this follows from Lemma \ref{lem:repthyquatgp} and the fact that the non-cyclic maximal subgroups of $Q_{2^{n}}$ are isomorphic to $Q_{2^{n-1}}$. Therefore we must have $P_{2}=Q_{2^{n}}$. If $K_{2}$ is a nontrivial normal subgroup of $Q_{2^{n}}$ then we have already observed that $\Def_{Q_{2^{n}}/K_{2}}^{Q_{2^{n}}}(\gamma_{n})=0$. Thus we must also have $K_{2}=\set{1}$. In particular, $P_{2}/K_{2}\iso Q_{2^{n}}$. But $P_{2}/K_{2}\iso P_{1}/K_{1}$ and the latter group is a subquotient of $Q_{2^{n-1}}$. This contradiction establishes that $F_{n}(Q_{2^{n-1}})=\set{0}$, from which it follows that $F_{n}$ is a proper subfunctor of $F_{n-1}$.
	
	To complete the proof it is enough to show that any nonzero subfunctor $F$ of $\overline{R}_{\Q}/R_{\Q}$ is equal to $F_{n}$ for some $n\geq 3$. So let $F$ be such a subfunctor. Then there exists a $2$-group $P$ such that $F(P)\neq\set{0}$. By rationality of $F$ there must exist a genetic subgroup $S$ of $P$ such that $F(N_{P}(S)/S)\neq\set{0}$ and $N_{P}(S)/S$ is generalized quaternion. In particular, there exists an integer $n\geq 3$ such that $F(Q_{2^{n}})\neq\set{0}$. Let $n$ be the smallest such integer. We show that $F=F_{n}$. Since $\gamma_{n}\in F(Q_{2^{n}})$ it is clear that $F_{n}\subseteq F$. Now let $P$ be an arbitrary $2$-group and let $\mathcal{G}$ be a genetic basis of $P$. Then, as in the remarks preceding the statement of the theorem, we have an isomorphism
	\begin{equation*}
	\underset{S\in\mathcal{Q}}{\Dsum}\Indinf_{N_{P}(S)/S}^{P}:\underset{S\in\mathcal{Q}}{\Dsum} F(N_{P}(S)/S)\isoto F(P),
	\end{equation*}
	where $\mathcal{Q}=\set{S\in\mathcal{G}:N_{P}(S)/S\iso Q_{2^{m}}\text{, some }m\geq 3}$. Observe that if $m$ is an integer in the range $3\leq m<n$ then $F_{n}(Q_{2^{m}})=\set{0}$ by the previous paragraph and $F(Q_{2^{m}})=\set{0}$ by minimality of $n$. If $m\geq n$ then $F_{n}(Q_{2^{m}})\neq\set{0}$, hence also $F(Q_{2^{m}})\neq\set{0}$. So we see that $F_{n}(Q_{2^{m}})=F(Q_{2^{m}})$ for all integers $m\geq 3$. It follows that the image of the map above is contained in $F_{n}(P)$ --- in other words, we have $F(P)\subseteq F_{n}(P)$. Since $P$ was an arbitrary $2$-group we conclude that $F\subseteq F_{n}$, hence $F=F_{n}$ as desired.
\end{proof}

Now for each integer $n\geq 3$ the quotient $F_{n}/F_{n+1}$ is a simple $2$-biset functor. In terms of Bouc's parametrization of simple biset functors (see \cite[Section~4.3]{Bouc_2010}) we have: 

\begin{corollary}
	For each integer $n\geq 3$ let $F_{n}$ denote the subfunctor of the $2$-biset functor $\overline{R}_{\Q}/R_{\Q}$ generated by $\gamma_{n}$. Then $F_{n}/F_{n+1}\iso S_{Q_{2^{n}},\F_{2}}$ where $\F_{2}$ is equipped with the trivial action of $\Out(Q_{2^{n}})$.
\end{corollary}

We also remark that Theorem \ref{thm:subfunctorsofRbarQ/RQ} can alternatively be obtained via Bouc's characterization of the subfunctors of rational $p$-biset functors, which is described in \cite[Section~10.2]{Bouc_2010}.

We conclude this section with a lemma that will be used in the proof of Theorem \ref{thm:subfunctorsofRK/imkappa}.

\begin{lemma}\label{lem:lemforlater}
	Let $P$, $\mathcal{G}$, and $\mathcal{Q}$ be as in the remarks preceding Theorem \ref{thm:subfunctorsofRbarQ/RQ}. For each $S\in\mathcal{Q}$ let $\gamma_{S}$ denote the unique nonzero element in $(\overline{R}_{\Q}/R_{\Q})(N_{P}(S)/S)$. The set $\set{\Indinf_{N_{P}(S)/S}^{P}(\gamma_{S}):S\in\mathcal{Q}}$ is an $\F_{2}$-basis of $(\overline{R}_{\Q}/R_{\Q})(P)$. For each integer $n\geq 3$ set
	\begin{equation*}
		\mathcal{Q}_{n}=\set{S\in\mathcal{G}:N_{P}(S)/S\iso Q_{2^{m}}\text{ for some }m\geq n}.
	\end{equation*}
	Then
	\begin{equation*}
		F_{n}(P)=\spn(\Indinf_{N_{P}(S)/S}^{P}(\gamma_{S}):S\in\mathcal{Q}_{n}).
	\end{equation*}	
\end{lemma}

\begin{proof}
	Let $\sum_{S\in\mathcal{Q}}a_{S}\Indinf_{N_{P}(S)/S}^{P}(\gamma_{S})$ be an arbitrary element of $F_{n}(P)$. If some coefficient $a_{T}\neq 0$ then
	\begin{equation*}
		\Defres_{N_{P}(T)/T}^{P}\left(\sum_{S\in\mathcal{Q}}a_{S}\Indinf_{N_{P}(S)/S}^{P}(\gamma_{S})\right)=a_{T}\gamma_{T}\in F_{n}(N_{P}(T)/T).
	\end{equation*} 
	Therefore $F_{n}(N_{P}(T)/T)\neq\set{0}$. Now if $N_{P}(T)/T\iso Q_{2^{m}}$ then the proof of Theorem \ref{thm:subfunctorsofRbarQ/RQ} shows that $n\leq m$. Since clearly $\Indinf_{N_{P}(S)/S}^{P}(\gamma_{S})\in F_{n}(P)$ for all $S\in\mathcal{Q}_{n}$ the equality holds.
\end{proof}

\section{The Subfunctors of $R_{K}/\im(\kappa)$}\label{sec:subfunctorsofRKmodimkappa}

Throughout this section $p$ denotes a fixed prime number and $K=\Q(\mu_{p'})$, where as always $\mu_{p'}$ denotes the group of complex roots of unity that have order relatively prime to $p$. Recall that $R_{K}(G)=\overline{R}_{K}(G)$ for any finite group $G$ --- this is Lemma \ref{lem:Brauerlem}.

Let $(\K,\OO,k)$ be a $p$-modular system such that $k$ is algebraically closed. Fix an embedding $K\into\K$. Recall the morphism of $\mu_{p'}$-fibered biset functors $\kappa:T_{\OO}\to R_{\K}$ introduced in Subsection \ref{subsec:examples}. Recall also that the image of $\kappa$ is a $\mu_{p'}$-fibered subfunctor of $R_{K}$. In this section we determine the subfunctor structure of the quotient $R_{K}/\im(\kappa)$.

We begin by recalling some well-known definitions from finite group theory. Let $q$ be a prime number and let $H=\gp{x}\semil Q$ be a semidirect product of a cyclic group $\gp{x}$ of order prime to $q$ by a $q$-group $Q$. Recall that such a group $H$ is called \textit{quasi-elementary for }$q$. If $Q$ centralizes $\gp{x}$, i.e., if $H=\gp{x}\times Q$, then $H$ is \textit{elementary for }$q$. We say that $H$ is $\Q(\mu_{p'})$\textit{-elementary for }$q$ if $Q$ centralizes the $p$-complement of $\gp{x}$ (this definition coincides with the standard --- and more general --- notion of $K$\textit{-elementary groups} given, for example, in \cite[Chapter~2, Section~21]{Curtis_1981_vol1} in the case where $K=\Q(\mu_{p'})$, which is the only case of importance in the sequel). Note that if $p$ does not divide the order of $H$, if $p=q$, or if $p=2$ then $H$ is $\Q(\mu_{p'})$-elementary for $q$ if and only if it is elementary for $q$. Finally, recall that a group $H$ is \textit{quasi-elementary} if it is quasi-elementary for some prime $q$. The classes of $\Q(\mu_{p'})$\textit{-elementary} and \textit{elementary} groups are defined analogously.

The proof of Theorem \ref{thm:alphaonto} requires a special case of the Witt-Berman Theorem, which is stated below. We remark that the Witt-Berman Theorem holds over arbitrary fields $K$ of characteristic 0, but we shall not make use of this fact.

\begin{theorem}\label{thm:wittberman}
	(Witt-Berman; \cite[Theorem~21.6]{Curtis_1981_vol1}) Let $p$ be a prime, let $K=\Q(\mu_{p'})$, and let $G$ be a finite group. If $\mathcal{H}^{p'}$ denotes the set of $\Q(\mu_{p'})$-elementary subgroups of $G$ then the homomorphism
	\begin{equation*}
		\underset{H\in\mathcal{H}^{p'}}{\Dsum}\Ind_{H}^{G}:\underset{H\in\mathcal{H}^{p'}}{\Dsum} R_{K}(H)\to R_{K}(G)
	\end{equation*}
	is surjective. 
\end{theorem}

Now the full subcategory of the $\mu_{p'}$-fibered biset category whose objects are the $p$-groups, $\CC_{p}^{\mu_{p'}}$, is a subcategory of the biset category $\CC$. Thus we have $\CC_{p}^{\mu_{p'}}=\CC_{p}$. Restriction from $\CC^{\mu_{p'}}$ to $\CC_{p}$ induces a functor $\RES:\FF^{\mu_{p'}}\to\FF_{p}$ with left adjoint $\IND:\FF_{p}\to\FF^{\mu_{p'}}$, as discussed in Subsection \ref{subsec:restrandind}. For any finite group $G$ we have
\begin{equation*}
\IND(\RES(R_{K}))(G)=\left(\underset{P\in \SS}{\bigoplus}B^{\mu_{p'}}(G,P)\tensor_{\Z} R_{K}(P)\right)/\RR_{G}
\end{equation*}
where $\SS$ is a fixed set of representatives for the isomorphism classes of $p$-groups and where $\RR_{G}$ is the subgroup generated by all elements of the form
\begin{equation*}
\left[\frac{G\times Q}{U,\phi}\right]\circ\left[\frac{Q\times P}{T,1}\right]\tensor\chi-\left[\frac{G\times Q}{U,\phi}\right]\tensor R_{K}\left(\left[\frac{Q\times P}{T,1}\right]\right)(\chi).
\end{equation*}
Here $P,Q\in\SS$, $\left[\frac{G\times Q}{U,\phi}\right]\in B^{\mu_{p'}}(G,Q)$ and $\left[\frac{Q\times P}{T,1}\right]\in B^{\mu_{p'}}(Q,P)$ are standard basis elements, and $\chi$ is a $K$-character of $P$. By Proposition \ref{prop:resandindadjunct} there is a bijection
\begin{equation*}
\psi:\Hom_{\FF_{p}}(\RES(R_{K}),\RES(R_{K}))\isoto\Hom_{\FF^{\mu_{p'}}}(\IND(\RES(R_{K})),R_{K}).
\end{equation*}
We set
\begin{equation*}
\alpha=\psi(\id_{\RES(R_{K})}).
\end{equation*}
At a finite group $G$ the natural transformation $\alpha$ has component
\begin{align*}
\alpha_{G}:\IND(\RES(R_{K}))(G)&\to R_{K}(G)\\
\left[\frac{G\times P}{U,\phi}\right]\tensor\chi&\mapsto R_{K}\left(\left[\frac{G\times P}{U,\phi}\right]\right)(\chi).
\end{align*}

\begin{lemma}\label{lem:imagesubring}
	Let $G$ be a finite group. Then $\im(\alpha_{G})$ is a unital subring of $R_{K}(G)$.
\end{lemma}

\begin{proof}
	It is clear that $\im(\alpha_{G})$ is a subgroup of $R_{K}(G)$ containing the principal character $1_{G}=\alpha_{G}(\Inf_{G/G}^{G}\tensor 1_{G/G})$, so we only need to show that $\im(\alpha_{G})$ is closed under multiplication. In fact, we need only show that
	\begin{equation}\label{eqn:prodalpha}
		\alpha_{G}\left(\left[\frac{G\times P}{U,\phi}\right]\tensor\chi\right)\cdot\alpha_{G}\left(\left[\frac{G\times Q}{T,\theta}\right]\tensor\psi\right)
	\end{equation}
	is contained in $\im(\alpha_{G})$ for any $p$-groups $P,Q\in\SS$, standard basis elements $\left[\frac{G\times P}{U,\phi}\right]\in B^{\mu_{p'}}(G,P)$, $\left[\frac{G\times Q}{T,\theta}\right]\in B^{\mu_{p'}}(G,Q)$, and characters $\chi\in R_{K}(P)$, $\psi\in R_{K}(Q)$. By the formula given in part 1 of \cite[Lemma~7.1.3]{Bouc_2010}, the product in (\ref{eqn:prodalpha}) above has character
	\begin{equation*}
		g\mapsto \frac{1}{|P||Q|}\sum_{\substack{x\in P\\y\in Q}}\Ind_{U}^{G\times P}(\phi)(g,x)\Ind_{T}^{G\times Q}(\theta)(g,y)\chi(x)\psi(y),\qquad g\in G.
	\end{equation*}
	We may assume that $P\times Q\in\SS$. If $\sigma:(G\times P)\times(G\times Q)\isoto(G\times G)\times(P\times Q)$ and $\delta:\Delta(G)\isoto G$ denote the obvious group isomorphisms then
	\begin{equation}\label{eqn:alphaGtimesG}
	\Iso(\delta)\Res_{\Delta(G)}^{G\times G}\left(\alpha_{G\times G}\left(\left[\frac{(G\times G)\times(P\times Q)}{\sigma(U\times T),(\phi\times\theta)\circ\sigma^{-1}}\right]\tensor(\chi\tensor\psi)\right)\right),
	\end{equation}
	which is an element of $\im(\alpha_{G})$, has character
	\begin{equation*}
	g\mapsto\frac{1}{|P||Q|}\sum_{\substack{x\in P\\y\in Q}}\Ind_{\sigma(U\times T)}^{(G\times G)\times(P\times Q)}((\phi\times\theta)\circ\sigma^{-1})(g,g,x,y)\chi(x)\psi(y),\qquad g\in G.
	\end{equation*}
	But since
	\begin{align*}
	\Ind_{\sigma(U\times T)}^{(G\times G)\times(P\times Q)}((\phi\times\theta)&\circ\sigma^{-1})(g,g,x,y)\\
	&=\Ind_{U}^{G\times P}(\phi)(g,x)\Ind_{T}^{G\times Q}(\theta)(g,y),
	\end{align*}
	the characters of (\ref{eqn:prodalpha}) and (\ref{eqn:alphaGtimesG}) are seen to be equal, and the proof is complete.
\end{proof}

\begin{lemma}\label{lem:relprimetensor}
	If $G$ and $H$ are finite groups whose orders are relatively prime then $R_{K}(G\times H)\iso R_{K}(G)\tensor_{\Z}R_{K}(H)$.
\end{lemma}

\begin{proof}
	We may assume that $p$ divides one of $|G|$ or $|H|$. Say $p$ divides $|G|$. If $\chi\in\Irr(G)$, $\psi\in\Irr(H)$ then $K(\chi\tensor\psi)=K(\chi)$ and the Galois class sum $\overline{\chi\tensor\psi}$, taken with respect to $K$, is equal to $\overline{\chi}\tensor\psi$. It follows that
	\begin{equation*}
		R_{K}(G\times H)=\overline{R}_{K}(G\times H)\iso\overline{R}_{K}(G)\tensor_{\Z}R_{\C}(H)=R_{K}(G)\tensor_{\Z}R_{K}(H).
	\end{equation*}
\end{proof}

\begin{theorem}\label{thm:alphaonto}
	The natural transformation $\alpha:\IND(\RES(R_{K}))\to R_{K}$, defined above, is surjective.
\end{theorem}

\begin{proof}
	Fix a finite group $G$. We must show that $\alpha_{G}$ is surjective. From the definition of $\alpha$ it is clear that $\alpha_{G}$ is surjective if $G$ is a $p$-group. On the other hand, if $G$ has order prime to $p$ then since $\im(\alpha_{G})$ contains all of the monomial characters $\Ind_{H}^{G}(\phi)$ for $H\subgp G$ and $\phi:H\to\mu_{p'}$, Brauer's Induction Theorem implies that $\alpha_{G}$ is surjective. Thus we may assume that the order of $G$ is divisible by $p$.
	
	Let $\mathcal{H}^{p'}$ denote the set of $\Q(\mu_{p'})$-elementary subgroups of $G$. Then by naturality of $\alpha$ there is a commutative diagram
	\begin{equation*}
	\begin{tikzcd}
	\underset{H\in\mathcal{H}^{p'}}{\bigoplus}\IND(\RES(R_{K}))(H) \arrow[d, "\dsum\alpha_{H}"'] \arrow[r, "\dsum\Ind_{H}^{G}"] & \IND(\RES(R_{K}))(G) \arrow[d, "\alpha_{G}"] \\
	\underset{H\in\mathcal{H}^{p'}}{\bigoplus}R_{K}(H) \arrow[r, "\dsum\Ind_{H}^{G}"', two heads]                                     & R_{K}(G).                                          
	\end{tikzcd}
	\end{equation*}
	The lower horizontal map is surjective by the Witt-Berman Theorem (see Theorem \ref{thm:wittberman}). Observe that if $\alpha_{H}$ is surjective for each $H\in\mathcal{H}^{p'}$ then $\alpha_{G}$ is surjective. Therefore we may assume that $G$ is $\Q(\mu_{p'})$-elementary for some prime number $q$. Then $G=\gp{x}\semil Q$ for some $q$-group $Q$ and an element $x$ of order prime to $q$ such that $Q$ centralizes the $p$-complement of $\gp{x}$.
	
	Suppose that $q=p$. Then $Q$ centralizes $\gp{x}$, so $G=\gp{x}\times Q$ and $R_{K}(G)\iso R_{K}(\gp{x})\tensor_{\Z}R_{K}(Q)$ by Lemma \ref{lem:relprimetensor}. A $\Z$-basis for $R_{K}(G)$ is formed by the set of characters $\lambda\tensor\overline{\chi}$ where $\lambda\in\Irr(\gp{x})$ and $\overline{\chi}\in R_{K}(Q)$ is a Galois class sum of irreducible complex characters of $Q$. Fix such a character $\lambda\tensor\overline{\chi}$. After identifying $Q$ with the quotient $G/\gp{x}$ we have
	\begin{equation*}
	(\Mult(\lambda\tensor 1_{Q})\circ\Inf_{Q}^{G})(\overline{\chi})=\Mult(\lambda\tensor 1_{Q})(1_{\gp{x}}\tensor\overline{\chi})=\lambda\tensor\overline{\chi},
	\end{equation*}
	where $1_{Q}$ and $1_{\gp{x}}$ denote the principal characters of $Q$ and $\gp{x}$, respectively. Since $\overline{\chi}\in\im(\alpha_{Q})$, the computation above shows that $\lambda\tensor\overline{\chi}\in\im(\alpha_{G})$. So $\alpha_{G}$ is surjective in this case.
	
	Suppose instead that $q\neq p$. Arguing as in the previous paragraph, we may assume without loss of generality that $\gp{x}$ is a $p$-group. Now if $p=2$ then $\Aut(\gp{x})$ is a $2$-group, so $G=\gp{x}\times Q$ and the characters $\overline{\lambda}\tensor\chi$, for $\overline{\lambda}\in R_{K}(\gp{x})$ a Galois class sum and $\chi\in\Irr(Q)$, form a $\Z$-basis of $R_{K}(G)$. Since $\alpha_{\gp{x}}$ and $\alpha_{Q}$ are surjective (as noted above) the inflated characters $\overline{\lambda}\tensor 1_{Q}$ and $1_{\gp{x}}\tensor\chi$ are both contained in $\im(\alpha_{G})$. Lemma \ref{lem:imagesubring} then implies that $\overline{\lambda}\tensor\chi\in\im(\alpha_{G})$, so $\alpha_{G}$ is surjective. Thus we may assume that $p\neq 2$. With this additional assumption, notice that the group $\Aut(\gp{x})$ is cyclic.
	
	Let $\chi\in\Irr(G)$ and let $\overline{\chi}\in\overline{R}_{K}(G)=R_{K}(G)$ be the corresponding Galois class sum. If we show that $\overline{\chi}\in\im(\alpha_{G})$ then, since $\chi$ was chosen arbitrarily, it will follow that $\alpha_{G}$ is surjective. Toward this end we may assume without loss of generality that $\chi$ is faithful. Let $\lambda\in\Irr(\gp{x})$ be an irreducible constituent of $\Res_{\gp{x}}^{G}(\chi)$. Then $\ker(\lambda)\nor G$ (in fact, every subgroup of $\gp{x}$ is normal in $G$). It follows that the kernel of $\lambda$ coincides with the kernel of $\Res_{\gp{x}}^{G}(\chi)$, hence $\ker(\lambda)=\set{1}$.
	
	Set 
	\begin{equation*}
		T=I_{G}(\lambda)=\set{g\in G:{}^{g}\lambda=\lambda},
	\end{equation*}
	the inertia group of $\lambda$ in $G$, and set $C=C_{Q}(x)$. Since $\lambda$ is faithful we have $T=C_{G}(x)=\gp{x}\times C$, hence $T\nor G$.
	
	Observe that the Galois conjugacy class of $\lambda$ over $K$ is equal to the set of faithful irreducible characters of $\gp{x}$, and that this set is stable under the action of $Q$ by conjugation. Let $\lambda=\lambda_{1},\lambda_{2},\ldots,\lambda_{n}$ denote representatives for the $Q$-orbits, and note that $I_{G}(\lambda_{i})=T$ for each $\lambda_{i}$. For each integer $i=1,2,\ldots,n$ set
	\begin{equation*}
		\Irr(G|\lambda_{i})=\set{\psi\in\Irr(G):\lambda_{i}\text{ is a constituent of }\Res_{\gp{x}}^{G}(\psi)}.
	\end{equation*}
	Define $\Irr(T|\lambda_{i})$ analogously. Recall (\cite[Theorem~6.11]{Isaacs_1976}) that induction from $T$ to $G$ yields a bijection
	\begin{equation}\label{eqn:irroverlambdabijection}
		\Ind_{T}^{G}:\Irr(T|\lambda_{i})\isoto\Irr(G|\lambda_{i}).
	\end{equation}
	Since $\Irr(T|\lambda)=\set{\lambda\tensor\beta:\beta\in\Irr(C)}$, it follows that there is a unique $\beta\in\Irr(C)$ such that $\chi=\Ind_{T}^{G}(\lambda\tensor\beta)$. Let $\beta=\beta_{1},\beta_{2},\ldots,\beta_{m}$ denote the distinct $Q$-conjugates of $\beta$, and note that $K(\chi)\subseteq K(\lambda\tensor\beta)=K(\lambda)$.
	
	Fix a pair $\lambda_{i},\beta_{j}$. Then $\beta_{j}={}^{u}\beta$ for some $u\in Q$ and ${}^{u^{-1}}\lambda_{i}={}^{\sigma}\lambda$ for some $\sigma\in\Gal(K(\lambda)/K)$. We compute:
	\begin{align*}
		\Ind_{T}^{G}(\lambda_{i}\tensor\beta_{j})&=\Ind_{T}^{G}(\lambda_{i}\tensor{}^{u}\beta)=\Ind_{T}^{G}({}^{u^{-1}}\lambda_{i}\tensor\beta)=\Ind_{T}^{G}({}^{\sigma}\lambda\tensor\beta)\\
			&=\Ind_{T}^{G}({}^{\sigma}(\lambda\tensor\beta))={}^{\sigma}\Ind_{T}^{G}(\lambda\tensor\beta)\\
			&={}^{\sigma}\chi.
	\end{align*}
	Thus $\Ind_{T}^{G}(\lambda_{i}\tensor\beta_{j})$ is Galois conjugate to $\chi$ over $K$. In particular, $\Ind_{T}^{G}(\lambda_{i}\tensor\beta_{j})\in\Irr(G)$. Since each $\sigma\in\Gal(K(\chi)/K)$ extends to an automorphism of $K(\lambda)$ the equalities above also show that every Galois conjugate of $\chi$ over $K$ is equal to $\Ind_{T}^{G}(\lambda_{i}\tensor\beta_{j})$ for some $\lambda_{i},\beta_{j}$.
	
	Now suppose that $\Ind_{T}^{G}(\lambda_{i}\tensor\beta_{j})=\Ind_{T}^{G}(\lambda_{k}\tensor\beta_{\ell})$ for some $1\leq i,k\leq n$, $1\leq j,\ell\leq m$. For ease, set $\psi=\Ind_{T}^{G}(\lambda_{i}\tensor\beta_{j})$. Then $\lambda_{i}$ and $\lambda_{k}$ are both constituents of $\Res_{\gp{x}}^{G}(\psi)$. Since $\psi$ is irreducible it follows that $\lambda_{i}$ and $\lambda_{k}$ are conjugate, hence $\lambda_{i}=\lambda_{k}$. In particular, $\lambda_{i}\tensor\beta_{j},\lambda_{k}\tensor\beta_{\ell}\in\Irr(T|\lambda_{i})$. The bijection (\ref{eqn:irroverlambdabijection}) above then implies that $\lambda_{i}\tensor\beta_{j}=\lambda_{k}\tensor\beta_{\ell}$.
	
	We have shown that the characters $\Ind_{T}^{G}(\lambda_{i}\tensor\beta_{j})$, $1\leq i\leq n$, $1\leq j\leq m$, are distinct and are precisely the Galois conjugates of $\chi$ over $K$. Therefore
	\begin{equation*}
		\overline{\chi}=\sum_{i=1}^{n}\sum_{j=1}^{m}\Ind_{T}^{G}(\lambda_{i}\tensor\beta_{j}).
	\end{equation*}
	Now because the quotient $Q/C$ is cyclic there exists a character $\theta\in R_{K}(Q)$ that extends $\sum_{j=1}^{m}\beta_{j}$. We compute that
	\begin{align*}
		\overline{\chi}&=\sum_{i=1}^{n}\Ind_{T}^{G}(\lambda_{i}\tensor\Res_{C}^{Q}(\theta))=\sum_{i=1}^{n}\Ind_{T}^{G}(\lambda_{i}\tensor 1_{C})\Inf_{Q}^{G}(\theta)\\
		&=\Ind_{T}^{G}\left(\sum_{i=1}^{n}\lambda_{i}\tensor 1_{C}\right)\Inf_{Q}^{G}(\theta).
	\end{align*}
	Notice that $\Inf_{Q}^{G}(\theta)\in\im(\alpha_{G})$. If $g\in G$ then
	\begin{equation*}
	\Ind_{T}^{G}\left(\sum_{i=1}^{n}\lambda_{i}\tensor 1_{C}\right)(g)=\begin{cases}
		(\overline{\lambda}\tensor 1)(g)	& \text{if }g\in T,\\
		0	& \text{if }g\notin T.
	\end{cases}
	\end{equation*}
	In the above, $\overline{\lambda}$ denotes the Galois class sum corresponding to $\lambda$, taken with respect to $K$. In other words, $\overline{\lambda}$ is the sum of the faithful irreducible complex characters of $\gp{x}$. The formula shows that $\Ind_{T}^{G}(\sum_{i=1}^{n}\lambda_{i}\tensor 1_{C})$ is a $K$-valued character of $G$ with kernel $C$, i.e., is an element of $R_{K}(G)$ that is inflated from $R_{K}(G/C)$. It follows that if $\alpha_{G/C}$ is surjective then $\Ind_{T}^{G}(\sum_{i=1}^{n}\lambda_{i}\tensor 1_{C})\in\im(\alpha_{G})$. Since $\im(\alpha_{G})$ is closed under multiplication by Lemma \ref{lem:imagesubring}, this in turn implies that $\overline{\chi}\in\im(\alpha_{G})$, as needed. 
	
	By the argument just given we may assume without loss of generality that $C=C_{Q}(x)=\set{1}$. Then $T=\gp{x}$. Let us say that $|x|=p^{a}$, where $a\geq 1$. We have $\overline{\chi}=\sum_{i=1}^{n}\Ind_{\gp{x}}^{G}(\lambda_{i})$, so $\overline{\chi}(1)=n|Q|=\varphi(p^{a})$ because $Q$ acts freely on the set of faithful irreducible complex characters of $\gp{x}$ (here $\varphi$ denotes Euler's totient function). Notice that $\Res_{Q}^{G}(\chi)=\Res_{Q}^{G}(\Ind_{\gp{x}}^{G}(\lambda))$ is equal to the regular character of $Q$, so $\chi$ is a constituent of $\Ind_{Q}^{G}(1_{Q})$. Since $\Ind_{Q}^{G}(1_{Q})$ is $K$-valued, it follows that $\overline{\chi}$ is a constituent of $\Ind_{Q}^{G}(1_{Q})$. Further, since the kernel of $\Ind_{\gp{x^{p^{a-1}}}Q}^{G}(1_{\gp{x^{p^{a-1}}}Q})$ contains the nontrivial subgroup $\gp{x^{p^{a-1}}}$ and since $\chi$, hence each Galois conjugate of $\chi$, is faithful, $\overline{\chi}$ is a constituent of the character
	\begin{equation*}
		\Ind_{Q}^{G}(1_{Q})-\Ind_{\gp{x^{p^{a-1}}}Q}^{G}(1_{\gp{x^{p^{a-1}}}Q}).
	\end{equation*}
	In fact, comparing degrees shows that $\overline{\chi}$ is equal to the character above. That $\overline{\chi}\in\im(\alpha_{G})$ now follows from Lemma \ref{lem:imagesubring}, which states in particular that $1_{H}\in\im(\alpha_{H})$ for all finite groups $H$. The proof is complete.
\end{proof}

We remark that the transformation $\alpha$ is not injective: if $G$ is a group whose order is not divisible by $p$ then $R_{K}(G)=R_{\C}(G)$ and there is an isomorphism $\IND(\RES(R_{K}))(G)\iso B^{\C^{\times}}(G)$ that makes the diagram below commute:
\begin{equation*}
	\begin{tikzcd}
	\IND(\RES(R_{K}))(G) \arrow[d, "\wr"'] \arrow[r, "\alpha_{G}"] & R_{\C}(G) \\
	B^{\C^{\times}}(G). \arrow[ru, "\lin_{G}"']                     &          
	\end{tikzcd}
\end{equation*}
Here $\lin_{G}$ denotes the \textit{linearization morphism} which maps $\left[\frac{G}{H,\phi}\right]\mapsto\Ind_{H}^{G}(\phi)$; see \cite[Section~11B]{Boltje_2018}.

Next we claim that there are equalities of $p$-biset functors
\begin{equation*}
	\RES(R_{K})=\overline{R}_{\Q}\qquad\text{and}\qquad\RES(\im(\kappa))=R_{\Q}.
\end{equation*}
Since $\im(\kappa)\subseteq R_{K}$ and $R_{\Q}\subseteq\overline{R}_{\Q}$ are subfunctors of $R_{\C}$, to prove the claim it suffices to show that $R_{K}(P)=\overline{R}_{\Q}(P)$ and $\im(\kappa_{P})=R_{\Q}(P)$ for any $p$-group $P$. So fix a $p$-group $P$ and recall that $R_{\C}(P)=\overline{R}_{\Q(\zeta_{p^{a}})}(P)$ for some primitive $p^{a}$-th root of unity $\zeta_{p^{a}}$. By Lemma \ref{lem:Brauerlem} we have $R_{K}(P)=\overline{R}_{K}(P)$, so
\begin{equation*}
R_{K}(P)=\overline{R}_{K}(P)=\overline{R}_{\Q(\zeta_{p^{a}})}(P)\cap\overline{R}_{K}(P)=\overline{R}_{\Q}(P).
\end{equation*}
On the other hand, from the description of $\im(\kappa)$ given in Example \ref{ex:trivsourcering} of Subsection \ref{subsec:examples} we can compute that
\begin{equation*}
\im(\kappa_{P})=\gp{\Ind_{Q}^{P}(1_{Q}):Q\subgp P}=R_{\Q}(P).
\end{equation*}
The last equality follows from the Ritter-Segal Theorem; see \cite{Ritter_1972} and \cite{Segal_1972}. 

Since $\RES(R_{K}/\im(\kappa))=\overline{R}_{\Q}/R_{\Q}$, by the adjunction of Proposition \ref{prop:resandindadjunct} there is a natural transformation $\overline{\alpha}:\IND(\overline{R}_{\Q}/R_{\Q})\to R_{K}/\im(\kappa)$ such that the diagram
\begin{equation*}
\begin{tikzcd}
\IND(\overline{R}_{\Q}) \arrow[d, "\alpha"', two heads] \arrow[r] & \IND(\overline{R}_{\Q}/R_{\Q}) \arrow[d, "\overline{\alpha}"] \\
R_{K} \arrow[r, two heads]                                        & R_{K}/\im(\kappa)                                            
\end{tikzcd}
\end{equation*}
commutes. Because $\alpha$ is surjective by Theorem \ref{thm:alphaonto}, $\overline{\alpha}$ is surjective as well. Combining this observation with Lemma \ref{lem:roquette} gives the following result.

\begin{corollary}\label{cor:quotientisF2space}
	We have $\im(\kappa)\neq R_{K}$ only when $p=2$, and in this case $R_{K}(G)/\im(\kappa_{G})$ is a (finite-dimensional) vector space over $\F_{2}=\Z/2\Z$ for every finite group $G$.
\end{corollary}

Our next aim is to classify the subfunctors of $R_{K}/\im(\kappa)$ in the case $p=2$. We require two lemmas.

\begin{lemma}\label{lem:robertsinductionlemma}
	Let $G$ be a finite group, let $\mathcal{H}$ be the set of quasi-elementary subgroups of $G$, and let
	\begin{equation*}
		\mathcal{U}=\set{U\subgp G:U\text{ has a normal Sylow }p\textit{-subgroup}}.
	\end{equation*}
	Then the principal character $1_{G}$ of $G$ is a $\Z$-linear combination of characters of the form $\Ind_{H}^{G}(\phi)$ where $H\in\mathcal{H}\cap\mathcal{U}$ and $\phi:H\to\mu_{p'}$. In particular, if $\theta\in R_{K}(G)$ then
	\begin{equation*}
		\theta=\sum a_{i}\Ind_{H_{i}}^{G}(\phi_{i}\Res_{H_{i}}^{G}(\theta))
	\end{equation*}
	for some $a_{i}\in\Z$, $H_{i}\in\mathcal{H}\cap\mathcal{U}$, and $\phi_{i}:H_{i}\to\mu_{p'}$.
\end{lemma}

\begin{proof}
	By \cite[Corollary~2]{Dress_1975} the isomorphism class of the trivial $\OO G$-module $\OO$ is a $\Z$-linear combination in $T_{\OO}(G)$ of (isomorphism classes of) modules of the form $\Ind_{U}^{G}(\OO_{\phi})$, where $U\in\mathcal{U}$ and $\phi:U\to\mu_{p'}$. Taking characters, we find that $1_{G}=\sum b_{i}\Ind_{U_{i}}^{G}(\phi_{i})$ for some integers $b_{i}$, $U_{i}\in\mathcal{U}$, and $\phi_{i}:U_{i}\to\mu_{p'}$. On the other hand, by Solomon's Induction Theorem (see \cite[Theorem~8.10]{Isaacs_1976}) we have $1_{G}=\sum c_{j}\Ind_{H_{j}}^{G}(1_{H_{j}})$ for some integers $c_{j}$ and subgroups $H_{j}\in\mathcal{H}$. We compute
	\begin{equation*}
		1_{G}=1_{G}\cdot 1_{G}=\sum_{i,j}b_{i}c_{j}\Ind_{U_{i}}^{G}(\phi_{i})\Ind_{H_{j}}^{G}(1_{H_{j}})
	\end{equation*}
	and, by applying Mackey's Theorem,
	\begin{equation*}
		\Ind_{U_{i}}^{G}(\phi_{i})\Ind_{H_{j}}^{G}(1_{H_{j}})=\sum_{g\in U_{i}\backslash G/H_{j}}\Ind_{{}^{g}H_{j}\cap U_{i}}^{G}(\Res_{{}^{g}H_{j}\cap U_{i}}^{U_{i}}(\phi_{i})).
	\end{equation*}
	Since $\mathcal{H}$ is stable under conjugation and since both $\mathcal{H}$ and $\mathcal{U}$ are closed under taking subgroups, ${}^{g}H_{j}\cap U_{i}\in\mathcal{H}\cap\mathcal{U}$. The first assertion has been proved, and the second follows from a short computation beginning with the equality $\theta=\theta\cdot 1_{G}$.
\end{proof}

\begin{lemma}\label{lem:RKforeltyfor2group}
	Suppose $p=2$ and let $H=\gp{x}\times P$ be elementary for $2$ (so $x$ has odd order and $P$ is a $2$-group). Then there is an $\F_{2}$-isomorphism
	\begin{equation*}
		R_{K}(H)/\im(\kappa_{H})\iso R_{\C}(\gp{x})\tensor_{\Z}(\overline{R}_{\Q}/R_{\Q})(P).
	\end{equation*}
\end{lemma}

\begin{proof}
Applying $R_{K}(\gp{x})\tensor_{\Z}$ to the short exact sequence $\im(\kappa_{P})\into R_{K}(P)\onto R_{K}(P)/\im(\kappa_{P})$ produces a short exact sequence
\begin{equation*}
	R_{K}(\gp{x})\tensor_{\Z}\im(\kappa_{P})\into R_{K}(\gp{x})\tensor_{\Z} R_{K}(P)\onto R_{K}(\gp{x})\tensor_{\Z}(R_{K}(P)/\im(\kappa_{P})).
\end{equation*}
Notice that the final group in the sequence above is equal to $R_{\C}(\gp{x})\tensor_{\Z}(\overline{R}_{\Q}/R_{\Q})(P)$. Now by Lemma \ref{lem:relprimetensor} there is a group isomorphism $R_{K}(\gp{x})\tensor_{\Z} R_{K}(P)\iso R_{K}(H)$. This isomorphism is part of a commutative diagram
\begin{equation*}
	\begin{tikzcd}
		T_{\OO}(\gp{x})\tensor_{\Z} T_{\OO}(P) \arrow[d, "\kappa_{\gp{x}}\tensor\kappa_{P}"'] \arrow[r] & T_{\OO}(H) \arrow[d, "\kappa_{H}"] \\
		R_{K}(\gp{x})\tensor_{\Z} R_{K}(P) \arrow[r, "\sim"]                                            & R_{K}(H)                          
	\end{tikzcd}
\end{equation*}
where the top horizontal map is induced by $[M]\tensor[N]\mapsto[M\tensor_{\OO}N]$. Surjectivity of this map follows from \cite[Theorem~1]{Dress_1975} and the assumption that $|x|$ and $|P|$ are coprime. In fact, the top horizontal map is an isomorphism: indeed, the Green Correspondence implies that the $\Z$-rank of $T_{\OO}(H)$ is equal to $|x|\cdot\rank_{\Z}T_{\OO}(P)$. By Brauer's Induction Theorem $\kappa_{\gp{x}}$ is surjective, hence bijective after comparing $\Z$-ranks. It follows that the $\Z$-ranks of $T_{\OO}(\gp{x})\tensor_{\Z} T_{\OO}(P)$ and $T_{\OO}(H)$ coincide and the top horizontal map is an isomorphism. In particular, the cokernels of the vertical maps above are isomorphic. The result follows after observing that $\im(\kappa_{\gp{x}}\tensor\kappa_{P})=R_{K}(\gp{x})\tensor_{\Z}\im(\kappa_{P})$.
\end{proof}

We now determine the subfunctor structure of $R_{K}/\im(\kappa)$ in the nontrivial case $p=2$. We continue to make use of the notation set in Lemma \ref{lem:repthyquatgp}. In particular,
\begin{equation*}
\gamma_{n}=\sum_{\substack{\chi\in\Irr(Q_{2^{n}})\\\text{faithful}}}\chi
\end{equation*}
is a representative for the unique nonzero coset in $(R_{K}/\im(\kappa))(Q_{2^{n}})$. We denote this coset, abusively, by $\gamma_{n}$.

The following theorem (and its proof) should be compared with Theorem \ref{thm:subfunctorsofRbarQ/RQ}.

\begin{theorem}\label{thm:subfunctorsofRK/imkappa}
	Set $p=2$. For each integer $n\geq 3$ let $F_{n}$ denote the subfunctor of the $\mu_{2'}$-fibered biset functor $R_{K}/\im(\kappa)$ generated by $\gamma_{n}$. Then the subfunctors $F_{n}$ are precisely the nonzero subfunctors of $R_{K}/\im(\kappa)$. We have
	\begin{equation*}
		R_{K}/\im(\kappa)=F_{3}\supsetneq F_{4}\supsetneq F_{5}\supsetneq\cdots.
	\end{equation*}
	In particular, $R_{K}/\im(\kappa)$ is generated by $\gamma_{3}$ and if $\im(\kappa_{G})\neq R_{K}(G)$ for some finite group $G$ then a Sylow $2$-subgroup of $G$ has a subquotient isomorphic to the quaternion group $Q_{8}$.
\end{theorem}

\begin{proof}
	By an argument similar to one employed in the proof of Theorem \ref{thm:subfunctorsofRbarQ/RQ} we have $F_{n}\subseteq F_{n-1}$ for each integer $n\geq 4$, and since $\RES(F_{n})$ is the subfunctor of the $2$-biset functor $\overline{R}_{\Q}/R_{\Q}$ generated by $\gamma_{n}$ each containment is proper: $F_{n}\subsetneq F_{n-1}$. Now $\overline{R}_{\Q}/R_{\Q}$ is generated (as a $2$-biset functor) by $\gamma_{3}$, so it is clear from the definition that $\IND(\overline{R}_{\Q}/R_{\Q})$ is generated (as a $\mu_{2'}$-fibered biset functor) by $\id_{Q_{8}}\tensor\gamma_{3}$. It follows that $R_{K}/\im(\kappa)$ is generated by $\overline{\alpha}(\id_{Q_{8}}\tensor\gamma_{3})=\gamma_{3}$, i.e., $R_{K}/\im(\kappa)=F_{3}$.
	
	Suppose that $\im(\kappa_{G})\neq R_{K}(G)$ for some finite group $G$. Then, by what we have just shown, there must exist a transitive $\mu_{2'}$-fibered $(G,Q_{8})$-biset $X$ such that
	\begin{equation*}
		(R_{K}/\im(\kappa))([X])(\gamma_{3})\neq 0.
	\end{equation*}  
	By \cite[Propositions~1.3 and 2.8]{Boltje_2018} there exist sections $(P_{1},K_{1})$ of $G$ and $(P_{2},K_{2})$ of $Q_{8}$ such that $P_{1}/K_{1}\iso P_{2}/K_{2}$ and $[X]$ factors through $\Defres_{P_{2}/K_{2}}^{Q_{8}}$ (note: this requires the fact that any homomorphism from a subgroup of $Q_{8}$ to $\mu_{2'}$ is trivial). Since $[X]$ does not annihilate $\gamma_{3}$ neither does $\Defres_{P_{2}/K_{2}}^{Q_{8}}$. Arguing as in the proof of Theorem \ref{thm:subfunctorsofRbarQ/RQ}, we find that $P_{2}=Q_{8}$ and $K_{2}=\set{1}$. It follows that $G$, hence a Sylow $2$-subgroup of $G$, has a subquotient isomorphic to $Q_{8}$. A similar argument shows that if $F_{n}(G)\neq\set{0}$ then $G$ has a subquotient isomorphic to $Q_{2^{n}}$. This implies in particular that $\cap_{n\geq 3}F_{n}=0$.
	
	Next we show that if $F$ is a subfunctor of $R_{K}/\im(\kappa)$ satisfying $F_{n}\subseteq F\subsetneq F_{n-1}$, some $n\geq 4$, then $F=F_{n}$. Let $F$ be such a subfunctor and suppose, by way of contradiction, that $F_{n}\subsetneq F$. Then there exists a group $G$ such that $F_{n}(G)\subsetneq F(G)$. Let $\theta\in F(G)\smallsetminus F_{n}(G)$. Then by Lemma \ref{lem:robertsinductionlemma} and by naturality of the projection $R_{K}\onto R_{K}/\im(\kappa)$ we may write
	\begin{equation}\label{eqn:thetaindmultres}
		\theta=\sum a_{i}(\Ind_{H_{i}}^{G}\circ\Mult(\phi_{i})\circ\Res_{H_{i}}^{G})(\theta)
	\end{equation}
	for some integers $a_{i}$, subgroups $H_{i}$ of $G$, and homomorphisms $\phi_{i}:H_{i}\to\mu_{2'}$. Moreover, we may assume that the $H_{i}$ are quasi-elementary groups that normalize their Sylow $2$-subgroups. In fact, we may assume that each $H_{i}$ is quasi-elementary for $2$ since if some $H_{i}$ is quasi-elementary for an odd prime then $H_{i}$ has no subquotient isomorphic to $Q_{8}$, hence $\im(\kappa_{H_{i}})=R_{K}(H_{i})$ and $\Res_{H_{i}}^{G}(\theta)=0$. But any subgroup $H_{i}$ that is quasi-elementary for $2$ and has a normal Sylow $2$-subgroup is elementary for $2$, so we may even assume that each subgroup $H_{i}$ is elementary for $2$. 
	
	Now suppose that $\Res_{H}^{G}(\theta)\in F_{n}(H)$ for all subgroups $H$ of $G$ that are elementary for $2$. Then in particular $\Res_{H_{i}}^{G}(\theta)\in F_{n}(H_{i})$ for each subgroup $H_{i}$ of the previous paragraph, so $\theta\in F_{n}(G)$ by equation (\ref{eqn:thetaindmultres}). Since we are assuming $\theta\notin F_{n}(G)$ there must exist a subgroup $H$ of $G$ that is elementary for $2$ and such that $\Res_{H}^{G}(\theta)\in F(H)\smallsetminus F_{n}(H)$. Write $H=\gp{x}\times P$ where $|x|$ is odd and $P$ is a $2$-group. Let $\mathcal{G}$ be a genetic basis of $P$ and for each $S\in\mathcal{G}$ such that $N_{P}(S)/S$ is generalized quaternion let $\gamma_{S}$ denote the unique nonzero element in $(\overline{R}_{\Q}/R_{\Q})(N_{P}(S)/S)$. Finally, for each integer $k\geq 3$ set
	\begin{equation*}
		\mathcal{Q}_{k}=\set{S\in\mathcal{G}:N_{P}(S)/S\iso Q_{2^{m}}\text{ for some }m\geq k}.
	\end{equation*}
	By Lemma \ref{lem:lemforlater} the set $\set{\Indinf_{N_{P}(S)/S}^{P}(\gamma_{S}):S\in\mathcal{Q}_{k}}$ is an $\F_{2}$-basis of $F_{k}(P)$. For simplicity, write $\overline{\chi}_{S}$ in place of $\Indinf_{N_{P}(S)/S}^{P}(\gamma_{S})$. By Lemma \ref{lem:RKforeltyfor2group} an $\F_{2}$-basis of $R_{K}(H)/\im(\kappa_{H})$ is given by the set $\set{\lambda\tensor\overline{\chi}_{S}:\lambda\in\Irr(\gp{x}),S\in\mathcal{Q}_{3}}$. We claim that
	\begin{equation*}
		F_{k}(H)=\spn(\lambda\tensor\overline{\chi}_{S}:\lambda\in\Irr(\gp{x}),S\in\mathcal{Q}_{k}).
	\end{equation*}
	To see why, first note that for any $\lambda\in\Irr(\gp{x})$ and $S\in\mathcal{Q}_{k}$ we have 
	\begin{equation*}
		\lambda\tensor\overline{\chi}_{S}=(\Mult(\lambda\tensor 1_{P})\circ\Inf_{P}^{H})(\overline{\chi}_{S})\in F_{k}(H).
	\end{equation*}
	On the other hand if $\sum a_{\lambda,S}\lambda\tensor\overline{\chi}_{S}\in F_{k}(H)$ --- where the sum is taken over all $\lambda\in\Irr(\gp{x})$ and $S\in\mathcal{Q}_{3}$ --- then for any $\mu\in\Irr(\gp{x})$ we have
	\begin{equation*}
		(\Def_{P}^{H}\circ\Mult(\mu^{-1}\tensor 1_{P}))(\sum_{\substack{\lambda\in\Irr(\gp{x}) \\ S\in\mathcal{Q}_{3}}}a_{\lambda,S}\lambda\tensor\overline{\chi}_{S})=\sum_{S\in\mathcal{Q}_{3}}a_{\mu,S}\overline{\chi}_{S}\in F_{k}(P).
	\end{equation*}
	Thus if $a_{\mu,S}\neq 0$ for some $S\in\mathcal{Q}_{3}$ we must have $S\in\mathcal{Q}_{k}$. This completes the proof of the claim. Now write
	\begin{equation*}
		\Res_{H}^{G}(\theta)=\sum_{\substack{\lambda\in\Irr(\gp{x}) \\ S\in\mathcal{Q}_{n-1}}}a_{\lambda,S}\lambda\tensor\overline{\chi}_{S}
	\end{equation*}
	for some $a_{\lambda,S}\in\F_{2}$. Since $\Res_{H}^{G}(\theta)\notin F_{n}(H)$ there must exist a character $\mu\in\Irr(\gp{x})$ and a genetic subgroup $T\in\mathcal{Q}_{n-1}\smallsetminus\mathcal{Q}_{n}$ such that $a_{\mu,T}\neq 0$. Notice that $N_{P}(T)/T\iso Q_{2^{n-1}}$. We have
	\begin{equation*}
		(\Defres_{N_{P}(T)/T}^{P}\circ\Def_{P}^{H}\circ\Mult(\mu^{-1}\tensor 1_{P})\circ\Res_{H}^{G})(\theta)=a_{\mu,T}\gamma_{T},
	\end{equation*}
	so $F(N_{P}(T)/T)\iso F(Q_{2^{n-1}})$ is nonzero. But then $\gamma_{n-1}\in F(Q_{2^{n-1}})$, i.e., $F_{n-1}\subseteq F$, which contradicts our assumption that $F\subsetneq F_{n-1}$. We conclude that $F_{n}$ is a maximal subfunctor of $F_{n-1}$ for all $n\geq 4$.
	
	Let $F$ be a nonzero subfunctor of $R_{K}/\im(\kappa)$. We complete the proof by showing that $F=F_{n}$ for some $n\geq 3$. Suppose first that $F(Q_{2^{n}})=\set{0}$ for all $n\geq 3$. Then $F\subseteq\cap_{n\geq 3}F_{n}$: if not, then there exists an integer $n\geq 4$ such that $F\subseteq F_{n-1}$ and $F\not\subseteq F_{n}$. We have $F_{n}\subsetneq F_{n}+F\subseteq F_{n-1}$, so $F_{n}+F=F_{n-1}$ since $F_{n}$ is a maximal subfunctor of $F_{n-1}$. But now
	\begin{equation*}
		\gamma_{n-1}\in F_{n-1}(Q_{2^{n-1}})=F_{n}(Q_{2^{n-1}})+F(Q_{2^{n-1}})=\set{0},
	\end{equation*}
	a contradiction. Thus $F\subseteq\cap_{n\geq 3}F_{n}$. Since $\cap_{n\geq 3}F_{n}=0$ (as noted above) we must have $F=0$, another contradiction. Therefore $F(Q_{2^{n}})\neq\set{0}$ for some integer $n$. Assume that $n$ is the minimum such integer. We show that $F=F_{n}$. It is clear that $F_{n}\subseteq F$ so we need only show that $F\subseteq F_{n}$. To this end, we prove by induction that $F\subseteq F_{k}$ for all integers $k$ in the range $3\leq k\leq n$. The base case $k=3$ is obvious: $F\subseteq R_{K}/\im(\kappa)=F_{3}$. Now if $F\subseteq F_{k}$ for some $k$, $3\leq k<n$, but $F\not\subseteq F_{k+1}$ then $F_{k}=F_{k+1}+F$, hence
	\begin{equation*}
		\gamma_{k}\in F_{k}(Q_{2^{k}})=F_{k+1}(Q_{2^{k}})+F(Q_{2^{k}})=F(Q_{2^{k}}).
	\end{equation*}
	But this contradicts the minimality of $n$. Thus if $F\subseteq F_{k}$ for some integer $k$ in the range $3\leq k<n$ then necessarily $F\subseteq F_{k+1}$. In particular $F\subseteq F_{n}$, as needed to complete the proof.
\end{proof}

In \cite[Theorem~9.2]{Boltje_2018}, Boltje and Co\c{s}kun parameterize the simple $A$-fibered biset functors $S_{(G,L,\lambda,V)}$ in terms of quadruples $(G,L,\lambda,V)$ where $G$ is a finite group, $(L,\lambda)$ is a \textit{reduced pair}, and $V$ is an irreducible module for a group ring $\Z\Gamma_{(G,L,\lambda)}$ (see \cite[Sections~8.1,~6.1]{Boltje_2018} for the respective definitions). Since, in the notation of the previous theorem, $F_{n}/F_{n+1}$ is a simple $\mu_{2'}$-fibered biset functor we have $F_{n}/F_{n+1}\iso S_{(G,L,\lambda,V)}$ for some such quadruple. Now the proof of \cite[Theorem~9.2]{Boltje_2018} shows that $G$ may be taken to be a group of smallest order for which $(F_{n}/F_{n+1})(G)\neq\set{0}$. Thus it is clear that we may take $G=Q_{2^{n}}$. By \cite[Corollary~10.13]{Boltje_2018}, the only reduced pair in this case is $(\set{1},1)$. Arguing as in the proof of \cite[Proposition~4.3.2]{Bouc_2010}, the group $\Gamma_{(Q_{2^{n}},\set{1},1)}$ is isomorphic to $\Out(Q_{2^{n}})$. The proof of \cite[Theorem~9.2]{Boltje_2018} then makes clear that we may take $V=\F_{2}$, where $V$ is given the unique $\Z\Out(Q_{2^{n}})$-module structure. Thus:
\begin{equation*}
	F_{n}/F_{n+1}\iso S_{(Q_{2^{n}},\set{1},1,\F_{2})}\qquad n\geq 3.
\end{equation*}

The corollary below follows directly from Lemma \ref{lem:robertsinductionlemma} and the arguments used in the proof of Theorem \ref{thm:subfunctorsofRK/imkappa}.

\begin{corollary}\label{cor:detectionthm}
	Let $G$ be a finite group and let $\chi\in R_{K}(G)$. Then $\chi\in\im(\kappa_{G})$ if and only if $\Res_{H}^{G}(\chi)\in\im(\kappa_{H})$ for all subgroups $H$ of $G$ that are elementary for 2 and possess a subquotient isomorphic to the quaternion group $Q_{8}$.
\end{corollary}


\bibliographystyle{plain}
\bibliography{../../../../Bibliography/bibliography}

\end{document}